\documentclass[a4paper]{amsart}
\usepackage[
left=24truemm,
right=24truemm
]{geometry}
\usepackage{amsfonts, amssymb, amsmath, verbatim, amsthm, mathrsfs, amscd, enumerate, ascmac}
\usepackage{fancyhdr, caption, color}
\usepackage{hyperref}
\numberwithin{equation}{section}
\usepackage{tikz}
\usetikzlibrary{positioning,intersections, calc, arrows,decorations.markings, angles}
\usepackage{pgfplots}
\theoremstyle{plain}
\newtheorem{theorem}{Theorem}[section]
\newtheorem{lemma}[theorem]{Lemma}
\newtheorem{proposition}[theorem]{Proposition}
\newtheorem{corollary}[theorem]{Corollary}
\theoremstyle{definition}

\theoremstyle{remark}

\def\d{\mathrm{d}}

\def\Re{\mathrm{Re}}
\def\BV{\mathrm{BV}}

\def\KG{\mathbf{KG}}

\def\F{\mathbf{F}}
\def\H{\mathbf{H}}
\def\W{\textbf{W}}
\def\KGpp{\mathbf{KG}_{(++)}}

\def\Fpp{\mathbf{F}_{(++)}}

\def\K{\mathcal K}
\def\L{\mathrm{L}}
\def\R{\mathrm{R}}

\def\setF{\mathfrak F}
\def\setG{\mathfrak G}
\def\setR{\mathfrak R}
\def\setN{\mathfrak N}

\author{Jayson Cunanan}
\address[Jayson Cunanan] {Department of Mathematics, Graduate school of Science and Engineering, Saitama University, Saitama, 338-8570, Japan}
\email{jcunanan@mail.saitama-u.ac.jp}
\author{Shobu Shiraki}
\address[Shobu Shiraki] {Department of Mathematics, Graduate school of Science and Engineering, Saitama University, Saitama, 338-8570, Japan}
\email{s.shiraki.446@ms.saitama-u.ac.jp}
\begin{document}
\date{\today}
\title[
Some sharp null-form type estimates for the Klein--Gordon equation
]
{
Some sharp null-form type estimates for the Klein--Gordon equation
}

\keywords{Bilinear estimates,  Klein--Gordon equation,  optimal constants,  extremisers,  Strichartz estimate.}
\subjclass[2010]{35B45,  42B37,  35A23.}
\begin{abstract}
We establish a sharp bilinear estimate for the Klein--Gordon propagator in the spirit of recent work of Beltran--Vega. Our approach is inspired by work in the setting of the wave equation due to Bez, Jeavons and Ozawa. As a consequence of our main bilinear estimate, we deduce several sharp estimates of null-form type and recover some sharp Strichartz estimates found by Quilodr\'an and Jeavons.
\end{abstract}
\maketitle
\section{Introduction}
Let $d\ge2$ and $\phi_s(r)=\sqrt{s^2+r^2}$ for $r\ge0$ and $s>0$. We write $D$ for the operator $\sqrt{-\Delta_x}$, that is 
\[
\widehat{Df}(\xi)=|\xi|\widehat{f}(\xi)
\]
where $\widehat{\cdot}$ denotes the (spatial) Fourier transform defined by
\[
\widehat{f}(\xi)=\int_{\mathbb R^d}e^{-ix\cdot\xi}f(x)\,\d x
\]
for appropriate functions $f$ on $\mathbb R^d$. Additionally, we define $D_\pm$ by
\[
\widetilde{D_\pm f}(\tau,\xi)=||\tau|\pm|\xi||\widetilde{f}(\tau,\xi),
\]
where $\widetilde{\cdot}$ is the space-time Fourier transform of appropriate functions $f$ on $\mathbb R\times \mathbb R^d$. The d'Alembertian operator $\partial_t^2-\Delta_x$ will be denoted by $\square$, so that $\square=D_-D_+$. Our main object of interest is the Klein--Gordon propagator given by
\[
e^{it\phi_s(D)} f(x)=\frac{1}{(2\pi)^d}\int_{\mathbb R^d} e^{i(x\cdot\xi+t\phi_s(|\xi|))}\widehat{f}(\xi)\,\d\xi
\]
for sufficiently nice initial data $f$.

As part of the study of sharp bilinear estimates for the Fourier extension operator and inspired by work of Ozawa--Tsutsumi \cite{OT98}, Beltran--Vega \cite{BV19} very recently presented the following sharp estimate associated to the Klein--Gordon propagator
\begin{align}\label{ineq:BV}
&\|D^{\frac{2-d}{2}}(e^{it\phi_s(D)}f\overline{e^{it\phi_s(D)}g})\|_{L^2(\mathbb R^d\times\mathbb R)}^2\\
&\quad\le
(2\pi)^{1-3d}
\int_{(\mathbb R^d)^2}
|\widehat{f}(\eta_1)|^2|g(\eta_2)|^2\phi_s(|\eta_1|)\phi_s(|\eta_2|)
K^\BV(\eta_1,\eta_2)
\,\d\eta_1\d\eta_2,\nonumber
\end{align}
where
\[
K^\BV(\eta_1,\eta_2)
=
\int_{\mathbb S^{d-1}} \frac{
\phi_s(|\eta_1|)+\phi_s(|\eta_2|)
}{
(\phi_s(|\eta_1|)+\phi_s(|\eta_2|))^2-((\eta_1+\eta_2)\cdot\theta)^2
}
\,\d\sigma(\theta).
\]
The estimate \eqref{ineq:BV} has some interesting connections to well-known results. For example, as we shall see in more detail later, \eqref{ineq:BV} leads to null-form type estimates by appropriately estimating the kernel. In particular, when $d=2$ the Strichartz estimate 
\begin{equation}\label{ineq:Strichartz, d=2}
\|
e^{it\phi_1(D)}f\|_{L^4(\mathbb R^{2+1})}
\leq
2^{-\frac14}
\|f\|_{H^\frac12(\mathbb R^2)}
\end{equation}
with the optimal constant follows from \eqref{ineq:BV}, more generally, for $d\geq2$ so does the null-form estimate
\begin{align}\label{ineq:Strichartz non-wave}
\|
D^{\frac{2-d}{2}}
|e^{it\phi_s(D)}f|^2\|_{L^2(\mathbb{R}^{d+1})}^2
\le
\frac{\pi^{2-d}|\mathbb S^{d-2}|}{2^ds}
\|\phi_s(D)^\frac12f\|_{L^2(\mathbb R^d)}^4,
\end{align}
although the optimality of the constant may be no longer true when $d\geq3$.
Here, the inhomogeneous Sobolev norm of order $\alpha$ is defined by
\[
\|f\|_{H^\alpha(\mathbb R^d)}:=\|\phi_1(D)^\alpha f\|_{L^2(\mathbb R^d)}.
\]
The estimate \eqref{ineq:Strichartz, d=2} with the optimal constant was first obtained by Quilodr\'an \cite{Qu15}. Bilinear estimates which bear resemblance to \eqref{ineq:BV} for the Klein--Gordon equation, as well as the Schr\"odinger and wave equation, have often arisen in the pursuit of optimal constants for Strichartz estimates and closely related null-form type estimates. As well as the aforementioned work of Beltran--Vega \cite{BV19}, estimates of the form \eqref{ineq:BV} for the Klein--Gordon propagator can be found in work of Jeavons \cite{Jv14} (see also \cite{Jvthesis}). For the Schr\"odinger equation, in addition to the Ozawa--Tsutsumi estimates in \cite{OT98}, estimates resembling \eqref{ineq:BV} may be found in work of Carneiro \cite{Cr09} and Planchon--Vega \cite{PV09}, with a unification of each of these results by Bennett et al. in \cite{BBJP17}. For the wave equation, Bez--Rogers \cite{BR13} and Bez--Jeavons--Ozawa \cite{BJO16} have established estimates resembling \eqref{ineq:BV}. We also remark that the related literature on sharp Strichartz estimates is large. In addition to the papers already cited, this body of work includes, for example, \cite{BJ15, COS19, COSS18, Fs07, HZ06, Ku03}; the interested reader is referred to the survey article by Foschi--Oliveira e Silva \cite{FO17} for further information.

Let $\Gamma(z)$ be the gamma function of $z$ (with $\Re(z)>0$) and
\begin{equation}
\K_a^b(\eta_1,\eta_2)
:=
\frac{\left(\phi_s(|\eta_1|)\phi_s(|\eta_2|)-\eta_1\cdot\eta_2-s^2\right)^b}
{\left(\phi_s(|\eta_1|)\phi_s(|\eta_2|)-\eta_1\cdot\eta_2+s^2\right)^a}.\nonumber
\end{equation}
In the present paper, we establish the following new bilinear estimates for the Klein--Gordon propagator. 

\begin{theorem}\label{thm:main}
For $d\ge2$ and $\beta>\frac{1-d}{4}$, we have the estimate
\begin{align}\label{ineq:main}
&\||\square|^\beta (e^{it\phi_s(D)}f\overline{e^{it\phi_s(D)}g})\|_{L^2(\mathbb{R}^{d+1})}^2\\
&\quad\le
\KG(\beta,d)
\int_{\mathbb{R}^{2d}}
|\widehat{f}(\eta_1)|^2|\widehat{g}(\eta_2)|^2\phi_s(|\eta_1|)\phi_s(|\eta_2|)
\K_\frac12^{\frac{d-2}{2}+2\beta}(\eta_1,\eta_2)
\,\d \eta_1\d \eta_2,\nonumber
\end{align}
with the optimal constant 
\[
\KG(\beta,d)
:=
2^{\frac{-5d+1}{2}+2\beta}\pi^{\frac{-5d+1}{2}}\frac{\Gamma(\tfrac{d-1}{2}+2\beta)}{\Gamma(d-1+2\beta)}.
\]
\end{theorem}
In the case when $s\to0$, certain sharp bilinear estimates for solutions to the wave equation with the operator $|\square|^\beta$ have been deeply studied by Bez--Jeavons--Ozawa \cite{BJO16}.  One may note that, when $d=2$, a slightly larger range of $\beta$ is valid in Theorem \ref{thm:main} than one for the corresponding result \eqref{ineq:biest sharp BJO} for the wave case in \cite{BJO16}.
In order to prove Theorem \ref{thm:main}, we employ their argument and adapt it into the context of the Klein--Gordon equation. As a consequence of Theorem \ref{thm:main}, we will generate null-form type estimates of the form
\begin{equation}\label{ineq:corollaries}
\||\square|^\beta |e^{it\phi_s(D)}f|^2\|_{L^2(\mathbb{R}^{d+1})}
\le
C
\|\phi_s(D)^\alpha f\|_{L^2(\mathbb R^d)}^2
\end{equation}
for certain pairs $(\alpha,\beta)$ with the optimal constant.

%
%
%
%
%
%

\subsection{Wave regime}
For $d\geq4$, the kernel $K^\BV$ can be estimated\footnote{
We observe that $d\geq4$ is important here. For $d=3$, the estimate \eqref{ineq:biest D} actually does not hold. The counterexample has been given by Foschi \cite{Fs00} for the wave equation, and the same argument appropriately adapted works for the Klein--Gordon propagator.
}
as
\begin{align*}
K^{\BV}(\eta_1,\eta_2)
&\leq
\frac{|\mathbb S^{d-1}|}{\phi_s(|\eta_1|)+\phi_s(|\eta_2|)}
\int_{-1}^1
\left(
1-\left|\frac{\eta_1+\eta_2}{\phi_s(|\eta_1|)+\phi_s(|\eta_2|)}\right|^2\lambda^2
\right)^{-1}
(1-\lambda^2)^{\frac{d-3}{2}}
\,\d\lambda\\
&\leq
\frac {C}{\phi_s(|\eta_1|)+\phi_s(|\eta_2|)}
\end{align*}
for some absolute constant $C$ since $|\eta_1+\eta_2|\leq\phi_s(|\eta_1|)+\phi_s(|\eta_2|)$. The integral in the first inequality is surely finite as long as $d\geq4$. Then, it follows from the arithmetic-geometric mean that
\[
\phi_s(|\eta_1|)\phi_s(|\eta_2|)K^{\BV}(\eta_1,\eta_2)
\leq
C\phi_s(|\eta_1|)^\frac12\phi_s(|\eta_2|)^\frac12,
\]
and hence the null-form type estimate
\begin{equation}\label{ineq:biest D}
\|D^{\frac{2-d}{2}}(e^{it\phi_s(D)}f\overline{e^{it\phi_s(D)}g})\|_{L^2(\mathbb R^d\times\mathbb R)}
\le
C
\|\phi_s(D)^\frac14f\|_{L^2(\mathbb R^d)}
\|\phi_s(D)^\frac14g\|_{L^2(\mathbb R^d)}
\end{equation}
holds. When $s\to0$, the estimate \eqref{ineq:biest D} yields
\begin{equation}\label{ineq:null +-}
\|D^{\beta_0}D_-^{\beta_-}D_+^{\beta_+}(e^{itD}f\overline{e^{itD}g})\|_{L^2(\mathbb R^{d+1})}
\leq
C\|f\|_{\dot{H}^{\alpha_1}}\|g\|_{\dot{H}^{\alpha_2}},
\end{equation}
in the case of $(\beta_0,\beta_-,\beta_+,\alpha_-,\alpha_+)=(\frac{2-d}{2},0,0,\frac14,\frac14)$ for the propagator $e^{itD}$ associated with the wave equation. The estimate \eqref{ineq:null +-}, as well as the corresponding $(++)$ case (while \eqref{ineq:null +-} is $(+-)$ case),
\begin{equation}\label{ineq:null ++}
\|D^{\beta_0}D_-^{\beta_-}D_+^{\beta_+}(e^{itD}fe^{itD}g)\|_{L^2(\mathbb R^{d+1})}
\leq
C\|f\|_{\dot{H}^{\alpha_1}}\|g\|_{\dot{H}^{\alpha_2}}
\end{equation}
has found important applications in study of nonlinear wave equations. This type of estimate has been studied back in work of Beals \cite{Be83} and  Klainerman--Machedon \cite{KM93,KM96a,KM97a}. A complete characterization of the admissible exponents $(\beta_0,\beta_-,\beta_+,\alpha_-,\alpha_+)$ for \eqref{ineq:null +-} and \eqref{ineq:null ++} was eventually obtained by Foschi--Klainerman \cite{FK00}. Such a characterization when the $L_{t,x}^2$ norm on the left-hand side of \eqref{ineq:null +-} is replaced by $L_t^qL_x^r$ has also drawn great attention. Using bilinear Fourier restriction techniques, Bourgain \cite{Br95} made a breakthrough contribution, then Wolff \cite{Wo01} and Tao \cite{Tao01} (in the endpoint case; see also Lee \cite{Lee06} and Tataru \cite{Tt07}) completed the diagonal case $q=r$. For the non-diagonal case we refer readers to \cite{LV08} due to Lee--Vargas for a complete characterization when $d\geq4$ and partial results when $d=2$, $3$. Soon later Lee--Rogers--Vargas \cite{LRV08} completed $d=3$, but a gap between necessary and sufficient conditions still remains when $d=2$.\\

As a means of comparing our bilinear estimate \eqref{ineq:main} with \eqref{ineq:BV}, we note that using the trivial bound
\begin{equation}\label{ineq:keyest wave}
\frac{\phi_s(|\eta_1|)\phi_s(|\eta_2|)-\eta_1\cdot\eta_2-s^2}{\phi_s(|\eta_1|)\phi_s(|\eta_2|)-\eta_1\cdot\eta_2+s^2}\le1,
\end{equation}
we estimate our kernel as 
\begin{equation}\label{ineq:kernelest wave}
\K_\frac12^{\frac{d-2}{2}+2\beta}(\eta_1,\eta_2)
\le
\K_0^{\frac{d-3}{2}+2\beta}(\eta_1,\eta_2).
\end{equation}
For $\beta\geq\frac{3-d}{4}$, it follows that 
\begin{align}\label{ineq:biest D_-D_+ }
\||\square|^{\beta}(e^{it\phi_s(D)}f \overline{e^{it\phi_s(D)}g})\|_{L^2(\mathbb{R}^{d+1})}
\le
C
\|\phi_s(D)^{\frac{d-1}{4}+\beta}f\|_{L^2(\mathbb{R}^d)}
\|\phi_s(D)^{\frac {d-1}{4}+\beta}g\|_{L^2(\mathbb{R}^d)}
\end{align}
for some absolute constant $C$ (for instance, use $\phi_s(|\eta_1|)\phi_s(|\eta_2|)-\eta_1\cdot \eta_2-s^2\leq 2\phi_s(|\eta_1|)\phi_s(\eta_2|)$), which, as in the discussion for the Beltran--Vega bilinear estimate, places Theorem \ref{thm:main} in the framework of null-form type estimates. If we \textit{formally} set $\beta=\frac{2-d}{4}$ in \eqref{ineq:biest D_-D_+ } to get data with regularity whose order is $\frac14$ as in \eqref{ineq:biest D}, the order of ``smoothing" from $|\square|^\beta$ becomes $2\beta=\frac{2-d}{2}$, which is compatible with \eqref{ineq:biest D}. Unfortunately, $\frac{2-d}{4}$ is outside the range $\beta\geq\frac{3-d}{4}$ and, in fact, as we shall see in Proposition \ref{prop:necessary condition}, $\beta\geq \frac{3-d}{4}$ is \emph{a necessary condition} for \eqref{ineq:biest D_-D_+ }. Nevertheless, as an application of Theorem \ref{thm:main} one can widen the range to, at least, $\beta\geq\frac{2-d}{4}$ if one considers radially symmetric data. We shall state this result as part of the forthcoming Corollary \ref{cor:biest radial}. In addition, for a large range of $\beta$ we shall in fact obtain the \emph{optimal constant} for such null-form type estimates; to state our result, we introduce the constant
\[
\F(\beta,d)
:=
2^{d-3+4\beta}\pi^{-\frac d2}\frac{\Gamma(\tfrac d2)\Gamma(\tfrac{d-1}{2}+2\beta)}{(d-2+2\beta)\Gamma(\tfrac{3d-5}{2}+2\beta)}.
\]

\begin{corollary}\label{cor:biest radial}
Let $d\ge2$, $\beta\geq\frac{2-d}{4}$. Then, there exists a constant $C>0$ such that \eqref{ineq:biest D_-D_+ } holds whenever $f$ and $g$ are radially symmetric.
Moreover, for $\beta\in[\frac{2-d}{4},\tfrac{3-d}{4}]\cup[\tfrac{5-d}{4},\infty)$, the optimal constant in \eqref{ineq:biest D_-D_+ } for radially symmetric $f$ and $g$ is $\F(\beta,d)^\frac12$, but there does not exist a non-trivial pair of functions $(f,g)$ that attains equality.
\end{corollary}

In the case of the wave propagator when $s\to0$, the estimate \eqref{ineq:biest D_-D_+ } becomes
\begin{equation}\label{ineq:Strichartz sharp BJO}
\||\square|^\beta(e^{itD}f\overline{e^{itD}g})\|_{L^2(\mathbb R^{d+1})}
\leq
\F(\beta,d)^\frac12
\|f\|_{\dot{H}^{\frac{d-1}{4}+\beta}(\mathbb R^d)}
\|g\|_{\dot{H}^{\frac{d-1}{4}+\beta}(\mathbb R^d)}
\end{equation}
and, in certain situation, it is known that the constant $\F(\beta,d)^\frac12$ is optimal. In the case $\beta=0$ and $d=3$, pioneering work of Foschi \cite{Fs07} established the optimality of the constant $\F(0,3)^\frac12$. The constant $\F(\beta,d)^\frac12$ is also known to be optimal when $(\beta,d)=(0,4)$ and $(\beta,d)=(0,5)$; the latter was established by Bez--Rogers \cite{BR13} building on work of Foschi and obtained via the bilinear estimate  
\begin{equation}\label{ineq:Strichartz sharp BJO beta=0}
\|e^{itD}f\overline{e^{itD}g}\|_{L^2(\mathbb R^{d+1})}^2
\le
\W(0,d)
\int_{(\mathbb R^{d})^2}|
\widehat{f}(\eta_1)|^2|\widehat{g}(\eta_2)|^2|\eta_1||\eta_2|
K_0^\mathrm{BR}(\eta_1,\eta_2)
\,\d\eta_1\d\eta_2.
\end{equation}
Here, $\W(\beta,d)$ turns out to be $\KG(\beta,d)$ and $K_{\beta}^\mathrm {BR}$ formally coincides with the special case of our kernel $\K_{\frac12}^{\frac{d-2}{2}+2\beta}$ when $s=0$. The optimality of $\F(0,4)^\frac12$ in \eqref{ineq:Strichartz sharp BJO} was proved by Bez--Jeavons \cite{BJ15} by making use of \eqref{ineq:Strichartz sharp BJO beta=0}, polar coordinates and techniques from the theory of spherical harmonics. 

Soon later, imposing an additional radial symmetry on the initial data $f$ and $g$, Bez--Jeavons--Ozawa proved the optimality of $\F(\beta,d)^\frac12$ in \eqref{ineq:Strichartz sharp BJO} when $d\geq2$ and\footnote{It can also be seen from the form of the sharp constant that widening the range $\beta>\beta_d$ is impossible.} $\beta>\beta_d:=\max\{\frac{1-d}{4},\frac{2-d}{2}\}$ by establishing the null-form type bilinear estimate 
\begin{equation}\label{ineq:biest sharp BJO}
\||\square|^\beta(e^{itD}f\overline{e^{itD}g})\|_{L^2(\mathbb R^{d+1})}^2
\le
\W(\beta,d)
\int_{(\mathbb R^{d})^2}|
\widehat{f}(\eta_1)|^2|\widehat{g}(\eta_2)|^2|\eta_1||\eta_2|
K_\beta^\mathrm{BR}(\eta_1,\eta_2)
\,\d\eta_1\d\eta_2,
\end{equation}
which again coincides\footnote{Because of this fact, we expect that \eqref{ineq:biest D_-D_+ } is valid with $C=\F(\beta,d)^\frac12$ for $\beta\in(\beta_d,\frac{2-d}{4})$ as well, but we do not pursue this here.} with \eqref{ineq:main} formally substituted with $s=0$.  They accomplished the result by taking advantage of an exceedingly nice structure of the homogeneity that the kernel $\K_\frac12^{\frac{d-2}{2}+2\beta}$ ($=K_\beta^\mathrm{BR}$) possesses, specifically, when $s=0$;
\[
\K_\frac12^{\frac{d-2}{2}+2\beta}(r_1\theta_2,r_2\theta_2)
=
(r_1r_2)^{\frac{d-3}{2}+2\beta}(1-\theta_1\cdot\theta_2)^{\frac{d-3}{2}+2\beta},\qquad r_1,r_2>0,\ \theta_1,\theta_2\in \mathbb S^{d-1},
\]
This property completely divides the right-hand side of \eqref{ineq:biest sharp BJO} into radial and angular components if the initial data are radial symmetric.
In contrast, our concern is the case $s>0$ and the lack of homogeneity in the kernel causes significant difficulty in this regard. This can be seen as responsible for the gap $(\frac{3-d}{4},\frac{5-d}{4})$ (as well as the range $(\beta_d,\frac{2-d}{4})$) in Corollary \ref{cor:biest radial}, for which we also expect \eqref{ineq:biest D_-D_+ } still holds with $C=\F(\beta,d)^\frac12$.\\

 In the current paper, we prove Corollary \ref{cor:biest radial} by first making use of our bilinear estimate \eqref{ineq:main}. One can show, however, that it is impossible to obtain the optimality of $\F(\beta,d)^\frac12$ in \eqref{ineq:biest D_-D_+ } for radial data and any $\beta\in(\frac{3-d}{4},\frac{5-d}{4})$ once one makes use of \eqref{ineq:main} as a first step; somewhat surprisingly given that \eqref{ineq:main} is sharp. This is a consequence of Lemma \ref{lem:B<B} in the specific setting.

There are some special cases of $\beta$; the endpoints $\frac{3-d}{4}$ and $\frac{5-d}{4}$ of the gap, at which we can remove the radial symmetry hypothesis on the initial data and still keep the optimal constants.

\begin{corollary}\label{cor:wave}
Let $d\ge2$. Then, the estimate \eqref{ineq:corollaries} holds with the optimal constant $C=\F(\beta,d)^\frac12$
for $(\alpha,\beta)=(\frac12,\frac{3-d}{4})$ and $(\alpha,\beta)=(1,\frac{5-d}{4})$, but there are no extremisers. Furthermore, when $(\alpha,\beta)=(1,\frac{5-d}{4})$, we have the refined Strichartz estimate
\begin{align}\label{ineq:Strichartz refined general J}
\||\square|^\frac{5-d}{4} |e^{it\phi_s(D)}f|^2\|_{L^2(\mathbb{R}^{d+1})}
\le
\F(\tfrac{5-d}{4},d)^\frac12
\left(
\|\phi_s(D)f\|_{L^2(\mathbb R^d)}^4
-s^2
\|\phi_s(D)^\frac12f\|_{L^2(\mathbb R^d)}^4
\right)^\frac12,
\end{align}
where the constant is optimal and there are no extremisers. 
\end{corollary}

Corollary \ref{cor:wave} generalizes the following recent results. In the context of the Klein--Gordon equation, Quilodr\'an \cite{Qu15} appropriately developed Foschi's argument in \cite{Fs07} and proved the sharp Strichartz estimate
\begin{equation}\label{ineq:Strichartz sharp Q}
\|e^{it\phi_1(D)}f\|_{L^q(\mathbb R^{d+1})}
\leq
\H(d,q)
\|f\|_{H^\frac12(\mathbb R^d)}
\end{equation}
for $(d,q)=(2,4)$, $(2,6)$, $(3,4)$, which are the endpoint cases of the admissible range of exponent $q$, namely, 
\[
\frac{2(d+2)}{d}\le q\le \frac{2(d+1)}{d-1}.
\]
The constant $\H(d,q)$ denotes the optimal constant so that \eqref{ineq:Strichartz sharp Q} in the case $(d,q)=(3,4)$ is recovered by Corollary \ref{cor:wave} in the case $(\alpha,\beta)=(\frac12,\frac{3-d}{4})$ and $\F(0,3)^\frac14=\H(4,3)$ holds. In \cite{Qu15}, Quilod\'an also proved that there is no extremiser which attains \eqref{ineq:Strichartz sharp Q} for $(d,q)=(2,4)$, $(2,6)$, $(3,4)$.
Later Carneiro--Oliveira e Silva--Sousa \cite{COS19} further revealed the nature of 
\eqref{ineq:Strichartz sharp Q}
for $d=1$, $2$, by answering questions raised in \cite{Qu15}; in particular, they found the best constant in \eqref{ineq:Strichartz sharp Q} for $(d,q)=(1,6)$ and absence of the extremisers (the case $(d,q)=(1,6)$ is the endpoint of the admissible range of $6\leq q\leq \infty$ when $d=1$). Meanwhile, they also established there exist extremisers in the non-endpoint cases in low dimensions $d=1$, $2$. A subsequent study by the same authors in collaboration with Stovall \cite{COSS18} proved the analogous results in the non-endpoint cases for higher dimensions $d\geq3$ by using some tools from bilinear restriction theory.

In \cite{Jv14}, Jeavons obtained the following refined Strichartz estimate in five spatial dimensions
\begin{align}\label{ineq:Strichartz refined J}
\|e^{it\phi_s(D)}f\|_{L^4(\mathbb R^{5+1})}
\leq
\F(0,5)^\frac14
\left(
\|\phi_s(D)f\|_{L^2(\mathbb R^5)}^4
-
s^2\|\phi_s(D)^\frac12f\|_{L^2(\mathbb R^5)}^4
\right)^\frac14,
\end{align}
which recovers the inequality \eqref{ineq:Strichartz sharp BJO} when $(\beta,d)=(0,5)$ in the limit $s\to0$. Moreover, by simply omitting the negative second term, it follows that
\begin{align*}
\|e^{it\phi_1(D)}f\|_{L^4(\mathbb{R}^{5+1})}
\le
\F(0,5)^\frac14
\|f\|_{H^1(\mathbb R^5)},
\end{align*}
where the constant $\F(0,5)^\frac14=(24\pi^2)^{-\frac14}$ is still sharp. These results are recovered too by Corollary \ref{cor:wave} in the case $(\alpha,\beta)=(1,\frac{5-d}{4})$.

\subsection{Non-wave regime}
One may examine the Beltran--Vega bilinear estimate \eqref{ineq:BV} from a somewhat different perspective to that taken in our earlier discussion which led to \eqref{ineq:biest D}. 
For $d\geq2$ the kernel $K^\BV$ can be transformed as
\begin{align*}
K^\BV(\eta_1,\eta_2)
&=
|\mathbb S^{d-2}|\int_0^{\pi}
\frac{
\phi_s(|\eta_1|)+\phi_s(|\eta_2|)
}{
(\phi_s(|\eta_1|)+\phi_s(|\eta_2|))^2-|\eta_1+\eta_2|^2\cos^2 \theta
}
(\sin\theta)^{d-2}
\,\d\theta\\
&=
|\mathbb S^{d-2}|\int_{-\frac\pi2}^{\frac\pi2}
\frac{
\tau(\cos\theta)^{d-2}
}{
(\tau^2-|\xi|^2)+|\xi|^2\cos^2\theta
}
\,\d\theta\\
&=
|\mathbb S^{d-2}|\int_{-\frac\pi2}^{\frac\pi2}
\frac{
\tau(\cos\theta)^{d-2}
}{
(\tau^2-|\xi|^2)\tan^2\theta+\tau^2
}
\frac{\d\theta}{\cos^2\theta},
\end{align*}
where we denote $\tau=\phi_s(|\eta_1|)+\phi_s(|\eta_2|)$ and $\xi=\eta_1+\eta_2$ for the sake of convenience. By applying the fact\footnote{Note that the equality holds when $d=2$} $(\cos \theta)^{d-2}\leq1$, the change of variables $\tan\theta\mapsto\frac{\tau}{\sqrt{\tau^2-|\xi|^2}}x$, it follows that
\[
K^\BV(\eta_1,\eta_2)
\leq
\frac{|\mathbb S^{d-2}|}{\sqrt{\tau^2-|\xi|^2}}\int_{-\infty}^\infty\frac{\d x}{x^2+1}
=
\frac{\pi|\mathbb S^{d-2}|}{\sqrt{\tau^2-|\xi|^2}}.
\]
Hence, another key relation (instead of \eqref{ineq:kernelest wave});
\begin{equation}\label{ineq:keyest non-wave}
\phi_s(|\eta_1|)\phi_s(|\eta_2|)-\eta_1\cdot\eta_2\geq s^2,
\end{equation}
implies
\[
K^\BV(\eta_1,\eta_2)
\leq 
\frac{\pi|\mathbb S^{d-2}|}{2s},
\]
with which, as informed earlier, the inequality \eqref{ineq:BV} directly yields \eqref{ineq:Strichartz non-wave}.
By comparison with \eqref{ineq:biest D}, the regularity level on the initial data has increased to $H^\frac12$ but this has allowed for a wider range of $d$ which, in particular, includes $d=2$ in which case \eqref{ineq:Strichartz non-wave} coincides with the sharp $H^\frac12\to L_{x,t}^4$ Strichartz estimate  \eqref{ineq:Strichartz, d=2} obtained by Quilodr\'an.
Note that, in the non-wave regime, we are not allowed to let $s\to0$ because of the factor $s^{-1}$ appearing in the constant.

On the other hand, Theorem \ref{thm:main} also yields \eqref{ineq:Strichartz, d=2} as a special case of the following family of sharp null-form type estimates valid in all dimensions $d\ge2$. Indeed, since we have another kernel estimate
\begin{equation}\label{ineq:kernelest non-wave}
\K_\frac12^{\frac{d-2}{2}+2\beta}(\eta_1,\eta_2)
\leq
2^{-\frac12}
\K_0^{\frac{d-2}{2}+2\beta}(\eta_1,\eta_2)
s^{-1}
\end{equation}
via \eqref{ineq:keyest non-wave}, we immediately deduce the following from Theorem \ref{thm:main}.

\begin{corollary}\label{cor:non-wave}
Let $d\ge2$. Then the estimate \eqref{ineq:corollaries} holds with the optimal constant 
$$
C
=
\left(\frac{2^{-d+1}\pi^{\frac{-d+2}{2}}}{s{\Gamma(\frac d2)}}\right)^\frac12
$$
for $(\alpha,\beta)=(\frac12,\frac{2-d}{4})$, but there are no extremisers. Furthermore, when $(\alpha,\beta)=(1,\frac{4-d}{4})$, we have the refined Strichartz estimate
\begin{align}\label{ineq:Strichartz refined non-wave}
\||\square|^\frac{4-d}{4} |e^{it\phi_s(D)}f|^2\|_{L^2(\mathbb{R}^{d+1})}^2
\le
\left(
\frac{2^{-d+1}\pi^{\frac{-d+2}{2}}}{s\Gamma(\frac{d+2}{2})}
\right)^\frac12
\left(
\|\phi_s(D)f\|_{L^2(\mathbb R^d)}^4
-s^2
\|\phi_s(D)^\frac12f\|_{L^2(\mathbb R^d)}^4
\right)^\frac12,
\end{align} 
where the constant is optimal and there are no extremisers.
\end{corollary}
One may note that \eqref{ineq:Strichartz refined non-wave} provides a sharp form of the following refined Strichartz inequality in the analogous manner of \eqref{ineq:Strichartz refined J} when $d=4$:
\begin{equation}\label{i:new}
\|e^{it\phi_1(D)}f\|_{L^4(\mathbb{R}^{4+1})}
\le
(16\pi)^{-\frac14}
(
\|f\|_{H^1(\mathbb R^4)}^4
-
\|f\|_{H^\frac12(\mathbb R^4)}^4
)^\frac14,
\end{equation}
however we are unable to conclude whether the constant $(16\pi)^{-\frac14}$ continues to be optimal if we drop the second term on the right-hand side, which is discussed in Section \ref{s:sharpness non-wave}. 
\\

For solutions $u$ of certain PDE, in addition to the null-form estimates \eqref{ineq:null +-}, estimates which control quantities like $|u|^2$ through its interplay with other types of operators have appeared numerous times in the literature. In particular, we note that the approach taken by Beltran--Vega \cite{BV19}, which in turn built on work of Planchon--Vega \cite{PV09}, rested on interplay with geometric operators such as the Radon transform or, more generally, the $k$-plane transform. For related work in this context of interaction with geometrically-defined operators, we also refer the reader to work of Bennett et. al \cite{BBFGI18} and Bennett--Nakamura \cite{BN20}.

Our approach to proving Theorem \ref{thm:main} more closely follows the argument in \cite{BJO16} and does not appear to fit into such a geometric perspective.\\

\subsection*{Summary of results}
Theorem \ref{thm:main}, a natural generalization of \eqref{ineq:biest sharp BJO} in the context of the Klein--Gordon, reproduces several known Strichartz-type inequalities with the sharp constant. The following is the summary of our results and remaining open problems.
\begin{itemize}
\item Corollary \ref{cor:wave} recovers \eqref{ineq:corollaries} when $(\beta,d)=(0,3)$ due to Quilodr\'an \cite{Qu15}.
\item Corollary \ref{cor:wave} recovers \eqref{ineq:corollaries} when $(\beta,d)=(0,5)$ due to Jeavons \cite{Jv14}.
\item Corollary \ref{cor:wave} recovers \eqref{ineq:Strichartz refined J} when $(\beta,d)=(0,5)$ with $\alpha=\frac34$ due to Jeavons \cite{Jv14}.
\item[$\circ$] For $(\beta,d)=(0,4)$, it remains open whether \eqref{ineq:corollaries} holds with $C=\F(0,4)^\frac12$ as the sharp constant.
\item Corollary \ref{cor:non-wave} recovers \eqref{ineq:corollaries} when $(\beta,d)=(0,2)$ due to Quilodr\'an \cite{Qu15}.
\item Corollary \ref{cor:non-wave} yields \eqref{i:new}, an analogous refined Strichartz inequality of \eqref{ineq:Strichartz refined J}, in the case $(\beta,d)=(0,4)$.
\item [$\circ$] Corollary \ref{cor:non-wave} recovers \eqref{ineq:corollaries} when $(\beta,d)=(0,4)$ with the constant $(16\pi)^{-\frac14}$, but we do not know whether the constant is sharp.
\end{itemize}

\subsection*{Notation/Useful formulae}
Throughout the paper, we denote $A\gtrsim B$ if $A\ge CB$, $A\lesssim B$ if $A\le CB$ and $A\sim B$ if $C^{-1}B\le A \le CB$ for some constant $C>0$.  The gamma function and the beta function are defined by
\[
\Gamma(z):=\int_0^\infty x^{z-1}e^{-x}\,\d x\quad \text{and} \quad B(z,w):=\int_0^1 \lambda^{z-1}(1-\lambda)^{w-1}\,\d\lambda,\]
respectively, for $z$, $w\in\mathbb C$ satisfying $\Re(z)$, $\Re(w)>0$. 
Regarding those, we use the following well-known formulae multiple times;
\begin{equation}\label{eq:volume of d-sphere}
|\mathbb{S}^{d-1}|=\frac{2\pi^{\frac{d}{2}}}{\Gamma(\frac{d}{2})},
\end{equation}
and
\[
B(z,w)=\frac{\Gamma(z)\Gamma(w)}{\Gamma(z+w)}.
\]
Also, it is worth to note here that the inverse Fourier transform of an appropriate function $g$ on $\mathbb R^d$ is given by $g^\vee(x)=(2\pi)^{-d}\int_{\mathbb R^d}e^{ix\cdot\xi}g(\xi)\,\d\xi$ so that the following hold:
\begin{itemize}
\item $\widehat{fg}(\xi)=(2\pi)^{-d}\widehat{f}*\widehat{g}(\xi),\quad \xi\in\mathbb R^d.$
\item $\|f\|_{L^2(\mathbb R^d)}^2=(2\pi)^{-d}\|\widehat{f}\|_{L^2(\mathbb R^d)}^2$ (Plancherel's theorem).
\item $\|\phi_s(D)^\alpha f\|_{L^2(\mathbb R^d)}
=
(2\pi)^{-d}\left(\int_{\mathbb R^d}\phi_s(|\xi|)^{2\alpha}|\widehat{f}(\xi)|^2\,\d\xi\right)^\frac12$.
\end{itemize}

\subsection*{Structure of the paper}
\begin{itemize}
\item Section \ref{sec:proof of Theorem main}: We first prove  Theorem \ref{thm:main} by adapting the argument of \cite{BJO16}.

\item Section \ref{sec:some observations wave regime}:  We prove \eqref{ineq:biest D_-D_+ } for radial data and then show that, specifically for $\beta\in[\frac{2-d}{4},\frac{3-d}{4}]\cup[\frac{5-d}{4},\infty)$, the estimate \eqref{ineq:biest D_-D_+ } with $C=\F(\beta,d)^\frac12$ holds. We also make an observation that suggests it may be difficult to obtain the optimal constant in \eqref{ineq:biest D_-D_+ } for $\beta\in(\frac{3-d}{4},\frac{5-d}{4})$ even for radially symmetric data (see Proposition \ref{prop:a counter example}).  At the end of this section, we show $\beta\geq\frac{3-d}{4}$ is necessary for \eqref{ineq:biest D_-D_+ } to hold for general data.

\item Section \ref{sec:sharpness of constants}: The aim of this section is to complete the proof of the corollaries. We first introduce how to deduce the refined form of the Strichartz estimate, and then focus on the sharpness of the constants in Corollary \ref{cor:biest radial}, Corollary \ref{cor:wave} and Corollary \ref{cor:non-wave}. They all are proved by the same method, but it differs from that in \cite{Jv14} or \cite{BJO16}, as here we need to deal with the more delicate situation of the non-wave regime. 
The non-existence of extremisers is also discussed. 

\item Section \ref{sec:analogous results ++}: We end the paper with Section \ref{sec:analogous results ++} by discussing analogous results for the $(++)$ case. As \cite{BJO16} has already observed, the $(++)$ case is far easier than the $(+-)$ case, and this will become clear from our argument in this section. We employ the null-form $|\square-(2s)^2|$ instead of $|\square|$ in order to follow the ideas of the proof of Theorem \ref{thm:main} and obtain an analogous bilinear estimate (Theorem \ref{thm:main ++}). 
\end{itemize}


\subsection*{Acknowledgment}
The first author was supported by JSPS Postdoctoral Research Fellowship (No. 18F18020),  and the second author was supported by JSPS KAKENHI Grant-in-Aid for JSPS Fellows (No.  20J11851).  Authors express their sincere gratitude to Neal Bez, second author's adviser, for introducing the problem, sharing his immense knowledge and continuous support. They also wish to thank the anonymous referee for a very careful reading of the manuscript and many valuable suggestions and comments.

\section{Proof of Theorem \ref{thm:main}}\label{sec:proof of Theorem main}

Although some steps require additional care due to the extra parameter $s$, broadly speaking Theorem \ref{thm:main} can be proved by adapting the argument for wave propagators presented in \cite{BJO16}, whose techniques originated in \cite{BBI15} (see also \cite{BBJP17}). The key tool here is the following Lorentz transform given by $L$; for fixed $(\tau,\xi)\in\mathbb R\times \mathbb R^d$ such that $\tau>|\xi|$, 
\[
L{t\choose x}={\gamma(t-\zeta\cdot x)\choose x+(\frac{\gamma-1}{|\zeta|^2}\zeta\cdot x-\gamma t)\zeta},\quad (t,x)\in\mathbb{R}\times\mathbb{R}^d,
\]
where $\zeta:=-\frac{\xi}{\tau}$ and $\gamma:=\frac{\tau}{(\tau^2-|\xi|^2)^{\frac12}}$. It is well known that the measure $\frac{\delta(\sigma-\phi_s(|\eta|))}{\phi_s(|\eta|)}$ for $(\sigma,\eta)\in\mathbb R\times\mathbb R^d$ is invariant under the Lorentz transform $L$, $|\det L|=1$, and
\begin{equation}\label{Lorentz}
L{(\tau^2-|\xi|^2)^{\frac12}\choose 0}={\tau\choose \xi}.
\end{equation}

Let us first introduce two lemmas whose proof come later in this section.

\begin{lemma}\label{lem:reformation of J}
For $\eta_1$, $\eta_2\in\mathbb{R}^{d}$ and $\beta>\frac{1-d}{4}$, define 
\begin{align}\label{eq:reformation of J}
J^{2\beta}(\eta_1,\eta_2)
:=
\int_{\mathbb{R}^{2d}}\frac{|\phi_s(|\eta_1|)\phi_s(|\eta_4|)-\eta_1\cdot\eta_4-s^2|^{2\beta}}{\phi_s(|\eta_3|)\phi_s(|\eta_4|)}\delta {\tau-\phi_s(|\eta_3|)-\phi_s(|\eta_4|)\choose \xi-\eta_3-\eta_4}\,\d\eta_3\d\eta_4,
\end{align}
where $\tau=\phi_s(|\eta_1|)+\phi_s(|\eta_2|)$ and $\xi=\eta_1+\eta_2$.
Then, we have
\[
J^{2\beta}(\eta_1,\eta_2)=(2\pi)^{\frac{d-1}{2}}\frac{\Gamma(\tfrac{d-1}{2}+2\beta)}{\Gamma(d-1+2\beta)}
\K_\frac12^{\frac{d-2}{2}+2\beta}(\eta_1,\eta_2).
\]
\end{lemma}
\begin{lemma}\label{lem:spherical rearrangement}
Let $\eta_1$, $\eta_2 \in\mathbb{R}^d$. Set 
\[
\xi=\eta_1+\eta_2,\qquad\tau=\phi_s(|\eta_1|)+\phi_s(|\eta_2|)
\]
and $\eta\in\mathbb{R}^d$ satisfying 
\[
2\phi_s(|\eta|)=(\tau^2-|\xi|^2)^{\frac12}.
\]
Then, 
there exists $\omega_*\in\mathbb{S}^{d-1}$ depending only on $\eta_1$, $\eta_2$  and $|\eta|$ such that 
\begin{align*}
{\phi_s(|\eta_1|)\choose-\eta_1}\cdot L{\phi_s(|\eta|)\choose\eta}-s^2&=|\eta|^2\left(1+\frac{\eta}{|\eta|}\cdot\omega_*\right).
\end{align*}
\end{lemma}

\begin{proof}[Proof of Theorem \ref{thm:main}]
Let $u(t,x)=e^{it\phi_s(D)}f(x)$ and $v(t,x)=e^{it\phi_s(D)}g(x)$. By the expressions $\widetilde{u}(\tau,\xi)
=
2\pi\delta(\tau-\phi_s(|\xi|))\widehat{f}(\xi)$ and $\widetilde{\overline{v}}(\tau,\xi)
=
2\pi\delta(\tau+\phi_s(|\xi|))\overline{\widehat{g}}(-\xi)$, Plancherel's theorem, and appropriately relabeling the variables, one can deduce
\begin{align*}
&
(2\pi)^{3d-1}
\||\Box|^\beta (u\overline{v})\|_{L^2(\mathbb{R}^{d+1})}^2\\
&\quad=
(2\pi)^{-4}
\int_{\mathbb{R}^{d+1}}|\tau^2-|\xi|^2|^{2\beta}|\widetilde{u}*\widetilde{\overline{v}}(\xi,\tau)|^2\,\d\tau\d\xi\\
&\quad=
\int_{\mathbb R^{4d}}\int_{\mathbb R^{d+1}}
|\tau^2-|\xi|^2|^{2\beta}\widehat{f}(\eta_1)\overline{\widehat{g}(-\eta_2)}
\overline{\widehat{f}(\eta_3)}\widehat{g}(-\eta_4)\\
&\quad\qquad\times
\delta{\tau-\phi_s(|\eta_1|)+\phi_s(|\eta_2|)\choose \xi-\eta_1-\eta_2}
\delta{\tau-\phi_s(|\eta_3|)+\phi_s(|\eta_4|)\choose \xi-\eta_3-\eta_4}\,\d\tau\d\xi\d\eta_1\d\eta_2\d\eta_3\d\eta_4\\
&\quad=
2^{2\beta}
\int_{\mathbb{R}^{4d}}|\phi_s(|\eta_1|)\phi_s(|\eta_4|)-\eta_1\cdot\eta_4-s^2|^{2\beta}
\frac{F(\eta_1,\eta_2)\overline{F(\eta_3,\eta_4)}}{(\phi_s(|\eta_1|)\phi_s(|\eta_2|)\phi_s(|\eta_3|)\phi_s(|\eta_4|))^{\frac12}}\\
&\quad\qquad\times
\delta {\phi_s(|\eta_1|)+\phi_s(|\eta_2|)-\phi_s(|\eta_3|)-\phi_s(|\eta_4|)\choose\eta_1+\eta_2-\eta_3-\eta_4}
\,\d\eta_1\d\eta_2\d\eta_3\d\eta_4.
\end{align*}
Here, the change of variables; $(\eta_2,\eta_4)\mapsto(-\eta_4,-\eta_2)$ has been performed in the last step and 
\[
F(\eta_1,\eta_2):=\widehat f(\eta_1)\widehat g(\eta_2)\phi_s(|\eta_1|)^\frac12\phi_s(|\eta_2|)^\frac12.
\]
If we define $\Psi=\Psi_s(\eta_1,\eta_2,\eta_3,\eta_4)=\left(\frac{\phi_s(|\eta_1|)\phi_s(|\eta_2|)}{\phi_s(|\eta_3|)\phi_s(|\eta_4|)}\right)^\frac12$, then by the arithmetic-geometric mean we have
\[
|F(\eta_1,\eta_2)F(\eta_3,\eta_4)|
\le\frac12\left(|F(\eta_1,\eta_2)|^2\Psi+|F(\eta_3,\eta_4)|^2\Psi^{-1}\right)
\]
so that
\begin{equation}\label{ineq:a/g mean}
\frac{|F(\eta_1,\eta_2)F(\eta_3,\eta_4)|}{(\phi_s(|\eta_1|)\phi_s(|\eta_2|)\phi_s(|\eta_3|)\phi_s(|\eta_4|))^{\frac12}}\le\frac12\left(\frac{|F(\eta_1,\eta_2)|^2}{\phi_s(|\eta_3|)\phi_s(|\eta_4|)}+\frac{|F(\eta_3,\eta_4)|^2}{\phi_s(|\eta_1|)\phi_s(|\eta_2|)}\right).
\end{equation}
The equality holds if and only if 
\[
\phi_s(|\eta_1|)\phi_s(|\eta_2|)\widehat{f}(\eta_1)\widehat{g}(\eta_2)
=
\phi_s(|\eta_3|)\phi_s(|\eta_4|)\widehat{f}(\eta_3)\widehat{g}(\eta_4)
\]
almost everywhere on the support of the delta measure, which is satisfied by, for instance, $f=g=f_a$ with $a>0$ that is given by
\begin{equation}\label{extremisers}
\widehat{f_a}(\xi)
=
\frac{e^{-a\phi_s(|\xi|)}}{\phi_s(|\xi|)}.
\end{equation}

Therefore, 
\begin{align*}
&\left(
(2\pi)^{-3d+1}2^{2\beta}
\right)^{-1}
\||\square|^\beta(u\overline{v})\|_{L^2(\mathbb R^{d+1})}^2\\
&\quad\leq
\frac12
\left[
\int_{\mathbb R^{2d}}F(\eta_1,\eta_2)J^{2\beta}(\eta_1,\eta_2)\,\d\eta_1\d\eta_2+\int_{\mathbb R^{2d}}F(\eta_2,\eta_1)J^{2\beta}(\eta_1,\eta_2)\,\d\eta_1\d\eta_2
\right],
\end{align*}
which implies \eqref{ineq:main} by applying Lemma \ref{lem:reformation of J}. One may note that the constant in \eqref{ineq:main} is sharp since we only apply the inequality \eqref{ineq:a/g mean} in the proof.
\end{proof}

We now prove the aforementioned lemmas.
\begin{proof}[Proof of Lemma \ref{lem:reformation of J}]
Let $\tau=\phi_s(|\eta_1|)+\phi_s(|\eta_2|)$ and $\xi=\eta_1+\eta_2$.  Recall the Lorentz transform $L$. 
The change of variables ${\sigma_j\choose\eta_j}\mapsto L{\sigma_j\choose\eta_j}$ for $j=3,4$ gives
\begin{align*}
J^{2\beta}(\eta_1,\eta_2)
&=
\int_{\mathbb{R}^{2(d+1)}}\left|{\phi_s(|\eta_1|)\choose-\eta_1}\cdot{\sigma_4\choose\eta_4}-s^2\right|^{2\beta}\\
&\qquad\times
\frac{\delta(\sigma_3-\phi_s(|\eta_3|))}{\phi_s(|\eta_3|)}\frac{\delta(\sigma_4-\phi_s(|\eta_4|))}{\phi_s(|\eta_4|)}\delta{\tau-\sigma_3-\sigma_4\choose\xi-\eta_3-\eta_4}\,\d\sigma_3\d\sigma_4\d\eta_3\d\eta_4\nonumber\\
&=
\int_{\mathbb{R}^{2(d+1)}}\left|{\phi_s(|\eta_1|)\choose-\eta_1}\cdot L{\sigma_4\choose\eta_4}-s^2\right|^{2\beta}\\
&\qquad\times
\frac{\delta(\sigma_3-\phi_s(|\eta_3|))}{\phi_s(|\eta_3|)}\frac{\delta(\sigma_4-\phi_s(|\eta_4|))}{\phi_s(|\eta_4|)}\delta{(\tau^2-|\xi|^2)^\frac12-\sigma_3-\sigma_4\choose \eta_3+\eta_4}\,\d\sigma_3\d\sigma_4\d\eta_3\d\eta_4\nonumber\\
&=
\int_{\mathbb{R}^{d}}\left|{\phi_s(|\eta_1|)\choose-\eta_1}\cdot L{\phi_s(|\eta|)\choose\eta}-s^2\right|^{2\beta}
\frac{1}{\phi_s(|\eta|)^2}\delta(2\phi_s(|\eta|)-(\tau^2-|\xi|^2)^{\frac12})\,\d\eta.
\end{align*}
By Lemma \ref{lem:spherical rearrangement} and switching to polar coordinates,
\begin{align*}
J^{2\beta}(\eta_1,\eta_2)
&=
\left(\int_{\mathbb{S}^{d-1}}(1+\theta\cdot\omega_*)^{2\beta}\,\d\sigma(\theta)\right)
\left(
\int_0^\infty
\frac{r^{4\beta}}{\phi_s(r)^2}\delta(2\phi_s(r)-(\tau^2-|\xi|^2)^{\frac12})r^{d-1}\,\d r
\right).
\end{align*}
The first integral can be further simplified as
\begin{align*}
\int_{\mathbb{S}^{d-1}}\left(1+\theta\cdot\omega_*\right)^{2\beta}\,\d\sigma(\theta)
&=
|\mathbb S^{d-2}|\int_{-1}^1(1+\lambda)^{2\beta}(1-\lambda^2)^{\frac{d-3}{2}}\,\d\lambda\\
&=
2^{d-2+2\beta} |\mathbb{S}^{d-2}|B \left(\tfrac{d-1}{2}+2\beta,\tfrac{d-1}{2}\right)
\end{align*}
by using the beta function $B$.
For the remaining radial integration, one can perform the change of variables $2\phi_s(r)\mapsto\nu$ in order to get
\begin{align*}
\int_0^\infty\frac{r^{4\beta}}{\phi_s(r)^2}\delta(2\phi_s(r)-(\tau^2-|\xi|^2)^{\frac12})r^{d-1}\,\d r
&=
2^{-d+2-4\beta}\int_{4s^2}^\infty\frac{(\nu^2-4s^2)^{\frac{d-2}{2}+2\beta}}{\nu}\delta(\nu-(\tau^2-|\xi|^2)^\frac12)\,\d\nu\\
&=
2^{\frac{-d+1}{2}-2\beta}
\K_\frac12^{\frac{d-2}{2}+2\beta}(\eta_1,\eta_2)
\end{align*}
and hence
\begin{align*}
J^{2\beta}(\eta_1,\eta_2)
=
2^{\frac{d-3}{2}} |\mathbb{S}^{d-2}|B (\tfrac{d-1}{2}+2\beta,\tfrac{d-1}{2})
\K_\frac12^{\frac{d-2}{2}+2\beta}(\eta_1,\eta_2).
\end{align*}
Finally, simplifying the constant by the formula
\[
 B (\tfrac{d-1}{2}+2\beta,\tfrac{d-1}{2})=\frac{\Gamma(\tfrac{d-1}{2}+2\beta)\Gamma(\tfrac{d-1}{2})}{\Gamma(d-1+2\beta)},
\]
we are done.
\end{proof}

\begin{proof}[Proof of Lemma \ref{lem:spherical rearrangement}]
Observe first that 
\[
L{\phi_s(|\eta|)\choose\eta}=\frac12{\tau+\frac{\xi\cdot\eta}{\phi_s(|\eta|)}\choose2\eta+\xi(1+\frac{\xi\cdot\eta}{(\tau+2\phi_s(|\eta|))\phi_s(|\eta|)})}.
\]
Then, a direct calculation gives
\begin{align*}
{\phi_s(|\eta_1|)\choose-\eta_1}\cdot L{\phi_s(|\eta|)\choose\eta}=
(\phi_s(|\eta|))^2\left(1+\frac{\eta}{|\eta|}\cdot|\eta|{z}\right),
\end{align*}
where
\[
z
=
\frac{[\phi_s(|\eta|)+\phi_s(|\eta_1|)]\eta_2-[\phi_s(|\eta|)+\phi_s(|\eta_2|)]\eta_1}{\phi_s(|\eta|)^2[\phi_s(|\eta_1|)+\phi_s(|\eta_2|)+2\phi_s(|\eta|)]}.
\]
Since we have the relation $2\phi_s(|\eta|)=\phi_s(|\eta_1|)\phi_s(|\eta_2|)-\eta_1\cdot\eta_2+s^2$, the numerator of $z$ can be simplified by
\begin{align*}
&|[\phi_s(|\eta|)+\phi_s(|\eta_1|)]\eta_2-[\phi_s(|\eta|)+\phi_s(|\eta_2|)]\eta_1|^2\\
&\quad=[\phi_s(|\eta|)+\phi_s(|\eta_1|)]^2|\eta_2|^2+[\phi_s(|\eta|)+\phi_s(|\eta_2|)]^2|\eta_1|^2-2[\phi_s(|\eta|)+\phi_s(|\eta_1|)][\phi_s(|\eta|)+\phi_s(|\eta_2|)]\eta_2\cdot\eta_1 \\
&\quad=[\phi_s(|\eta|)+\phi_s(|\eta_1|)]^2\phi_s(|\eta_2|)^2+[\phi_s(|\eta|)+\phi_s(|\eta_2|)]^2\phi_s(|\eta_1|)^2\\
&\quad\qquad-2[\phi_s(|\eta|)+\phi_s(|\eta_2|)][\phi_s(|\eta|)+\phi_s(|\eta_1|)](\phi_s(|\eta_1|)\phi_s(|\eta_2|)-2\phi_s(|\eta|)^2)\\
&\quad\qquad\qquad-s^2([\phi_s(|\eta|)+\phi_s(|\eta_1|)]^2+[\phi_s(|\eta|)+\phi_s(|\eta_2|)]^2+2[\phi_s(|\eta|)+\phi_s(|\eta_1|)][\phi_s(|\eta|)+\phi_s(|\eta_2|)])\\
&\quad=\left(\phi_s(|\eta|)^2-s^2\right)[\phi_s(|\eta_1|)+\phi_s(|\eta_2|)+2\phi_s(|\eta|)]^2,
\end{align*}
and so it follows that
\[
|z|=\frac{|\eta|}{\phi_s(|\eta|)^2}.
\]
Therefore, 
\begin{align*}
{\phi_s(|\eta_1|)\choose-\eta_1}\cdot L{\phi_s(|\eta|)\choose\eta}-s^2
=|\eta|^2\left(1+\frac{\eta}{|\eta|}\cdot\omega_*\right),
\end{align*}
where we have set $\omega_*=\frac{z}{|z|}$.
\end{proof}


\section{On estimate \eqref{ineq:biest D_-D_+ }}\label{sec:some observations wave regime}
As announced in the introduction, we focus on Corollary \ref{cor:biest radial} by assuming the radial symmetry on initial data. Firstly, we prove the claimed inequality \eqref{ineq:biest D_-D_+ } with the constant $C=\F(\beta,d)$ by applying Theorem \ref{thm:main}. Here, to  deal with the complexity raised in the context of the Klein--Gordon equation, deriving a monotonicity property of the corresponding kernel is useful (Lemma \ref{lem:monotonicity}). Secondly, we show that there is, unfortunately, no way to conclude the same as Corollary \ref{cor:biest radial}, via Theorem \ref{thm:main}, for $\beta$ in the gap $(\frac{3-d}{4},\frac{5-d}{4})$.  Thirdly, we observe a phenomenon that the radial symmetry on the initial data allows the inequality to hold with a wider range of $\beta$, namely, its lower bound from $\beta>\frac{3-d}{4}$ to, at least, $\beta\geq\frac{2-d}{4}$.

\subsection{Estimate \eqref{ineq:biest D_-D_+ } with explicit constant}
Here, we prove \eqref{ineq:biest D_-D_+ } for radially symmetric data $f$ and $g$ for $\beta\geq\frac{2-d}{4}$ and an explicit constant $C<\infty$; for $\beta=[\frac{2-d}{4},\frac{3-d}{4}]\cup[\frac{5-d}{4},\infty)$, this explicit constant coincides with $\F(\beta,d)^\frac12$. In order to complete the proof of Corollary \ref{cor:biest radial}, we need to show the sharpness of $\F(\beta,d)^\frac12$ for $\beta\in[\frac{2-d}{4},\frac{3-d}{4}]\cup[\frac{5-d}{4},\infty)$, and the non-existence of extremisers; for these arguments, we refer the reader to Section \ref{sec:sharpness of constants}.

\begin{lemma}\label{lem:monotonicity}
Let $a+b>-1$, $b>-1$ and $\kappa\in[0,1]$. Define
\[
h^{a,b}(\kappa)
:=
\int_{-1}^1(1-\kappa\lambda)^a(1-\lambda^2)^b\,\d\lambda.
\]
Then,
\[
\sup_{\kappa\in[0,1]}h^{a,b}(\kappa)<\infty.
\]
Moreover, for $a\in(-\infty,0]\cup[1,\infty)$
\[
\sup_{\kappa\in[0,1]}h^{a,b}(\kappa)
=
h^{a,b}(1)
=
2^{a+2b+1}B(a+b+1,b+1).
\]
\end{lemma}

\begin{proof}
By the Lebesgue dominated convergence theorem,
\begin{align*}
\frac{\d}{\d\kappa}h^{a,b}(\kappa)&=-a\kappa\int_{-1}^1(1-\kappa\lambda)^{a-1}\lambda(1-\lambda^2)\,\d\lambda\\
&=a\kappa\int_0^1\left((1+\kappa\lambda)^{a-1}-(1-\kappa\lambda)^{a-1}\right)\lambda(1-\lambda^2)^b\,\d\lambda
\end{align*}
Thus,
\[
\begin{cases}
\frac{\d}{\d\kappa}h^{a,b}(\kappa)\ge0\qquad&\text{if $a\in(-\infty,0]\cup[1,\infty)$},\\
\frac{\d}{\d\kappa}h^{a,b}(\kappa)<0\qquad&\text{if $a\in(0,1)$}.
\end{cases}
\]

For $a\in(-\infty,0]\cup[1,\infty)$,
\[
\sup_{\kappa\in[0,1]}h^{a,b}(\kappa)
=
h^{a,b}(1)=\int_{-1}^1(1-\lambda)^a(1-\lambda^2)^b\,\d\lambda
\]
and the change of variables $1+\lambda\mapsto2\lambda$ gives
\[
\int_{-1}^1(1-\lambda)^a(1-\lambda^2)^b\,\d\lambda=2^{a+2b+1}B(a+b+1,b+1)<\infty
\]
if $a+b>0$ and $b>-1$. Similarly, for $a\in(0,1)$,
\[
h^{a,b}(\kappa)\le h^{a,b}(0)=2^{2b+1}B(b+1,b+1)<\infty
\]
if $b>-1$.
\end{proof}

\begin{proof}[Proof of Corollary \ref{cor:biest radial}]
Let $f$, $g$ be radially symmetric. By Theorem \ref{thm:main}, we have
\begin{align}\label{i:radial theorem}
&\||\square|^{\beta}(e^{it\phi_s(\sqrt{-\Delta})}f\overline{e^{it\phi_s(\sqrt{-\Delta})}g})\|_{L^2(\mathbb{R}^{d+1})}^2\nonumber\\
&\quad\le \KG(\beta,d)
\int_0^\infty\!\int_0^\infty
|\widehat{f}(r_1)|^2|\widehat{g}(r_2)|^2
\phi_s(r_1)^{\frac{d-1}{2}+2\beta}\phi_s(r_2)^{\frac{d-1}{2}+2\beta}\Theta_\frac12^{\frac{d-2}{2}+2\beta}(r_1,r_2)r_1^{d-1}r_2^{d-1}\,\d r_1\d r_2,
\end{align}
where
\[
\Theta_a^b(r_1,r_2)
:=
\int_{(\mathbb{S}^{d-1})^2}
\frac{
\left(1-\frac{r_1r_2\theta_1\cdot\theta_2}{\phi_s(r_1)\phi_s(r_2)}-\frac{s^2}{\phi_s(r_1)\phi_s(r_2)}\right)^b
}{
\left(1-\frac{r_1r_2\theta_1\cdot\theta_2}{\phi_s(r_1)\phi_s(r_2)}+\frac{s^2}{\phi_s(r_1)\phi_s(r_2)}\right)^a
}
\,\d\sigma(\theta_1)\d\sigma(\theta_2).
\]
We divide the range of $\beta$ into $\beta\in[\frac{2-d}{4},\frac{3-d}{4}]$ and $\beta\in[\frac{5-d}{4},\infty)$ and treat these cases differently. First, let us consider $\beta\in[\frac{5-d}{4},\infty)$ as the easier case. By applying the fundamental kernel estimate \eqref{ineq:kernelest wave}, we have
\begin{align*}
\Theta_\frac12^{\frac{d-2}{2}+2\beta}(r_1,r_2)
&\leq
\Theta_0^{\frac{d-3}{2}+2\beta}(r_1,r_2)=|\mathbb S^{d-1}||\mathbb S^{d-2}|h^{\frac{d-3}{2}+2\beta,\frac{d-3}{2}}(\kappa)
\end{align*}
with $\kappa=\frac{r_1r_2}{\phi_s(r_1)\phi_s(r_2)}$. Since $d-3+2\beta\geq1$, Lemma \ref{lem:monotonicity} implies that 
\[
\sup_{\kappa\in[0,1)}h^{\frac{d-3}{2}+2\beta,\frac{d-3}{2}}(\kappa)
=
h^{\frac{d-3}{2}+2\beta,\frac{d-3}{2}}(1),
\]
and hence
\[
\sup_{r_1,r_2>0}
\Theta_\frac12^{\frac{d-2}{2}+2\beta}(r_1,r_2)\le 2^{\frac{3d-7}{2}+2\beta}|\mathbb{S}^{d-1}||\mathbb{S}^{d-2}|B\left(d-2+2\beta,\tfrac{d-1}{2}\right),
\]
which yields \eqref{ineq:biest D_-D_+ } with $C=\F(\beta,d)^\frac12$.\\

For $\beta\in[\frac{2-d}{4},\frac{3-d}{4}]$, in which case $ \frac{d-2}{2}+2\beta\in[0,\frac12]$, the basic idea of our argument is the same as above but it requires a few more steps.  Let
\[
\Xi(\nu,\upsilon)
:=
\int_{-1}^1
\frac{
(
1-\nu-\sqrt{1-\nu^2-\upsilon^2}\lambda
)^{\frac{d-2}{2}+2\beta}
}{
(
1+\nu-\sqrt{1-\nu^2-\upsilon^2}\lambda
)^\frac12
}
\,\d\mu(\lambda)
\]
with $\nu$ and $\upsilon$ satisfying
\[
\nu\in[0,1],\qquad \upsilon^2\leq1-\nu^2,
\]
and $\d\mu(\lambda)=(1-\lambda^2)^\frac{d-3}{2}\,\d\lambda$. Then, from \eqref{i:radial theorem}, it suffices to show
\begin{equation}\label{i:Xi Xi Xi(0,0)}
\Xi(\nu,\upsilon)\leq\Xi(0,\upsilon)\leq\Xi(0,0).
\end{equation}
In order to show the first inequality of \eqref{i:Xi Xi Xi(0,0)}, we establish monotonicity in $\nu \in \big[0,\sqrt{\frac{1-\upsilon^2}{2}}\big]$, and calculate directly for $\nu\in\big[\sqrt{\tfrac{1-\upsilon^2}{2}},\sqrt{1-\upsilon^2}\big]$.
Indeed, it simply follows that 
\begin{align*}
\partial_\nu\Xi(\nu,\upsilon)&\leq
-
\left(\frac{d-2}{2}+2\beta\right)
\int_0^1
\frac{(1-\nu-\sqrt{1-\nu^2-\upsilon^2}\lambda)^{\frac{d-4}{2}+2\beta}}{(1+\nu-\sqrt{1-\nu^2-\upsilon^2}\lambda)^\frac12}\left(1-\frac{\nu}{\sqrt{1-\nu^2-\upsilon^2}}\lambda\right)\,\d\mu(\lambda)\\
&\qquad
-
\frac12
\int_0^1
\frac{(1-\nu-\sqrt{1-\nu^2-\upsilon^2}\lambda)^{\frac{d-2}{2}+2\beta}}{(1+\nu-\sqrt{1-\nu^2-\upsilon^2}\lambda)^\frac32}\left(1+\frac{\nu}{\sqrt{1-\nu^2-\upsilon^2}}\lambda\right)\,\d\mu(\lambda)\\
&\qquad\quad
-
\left(\frac{d-2}{2}+2\beta\right)
\int_0^1
\frac{(1-\nu+\sqrt{1-\nu^2-\upsilon^2}\lambda)^{\frac{d-4}{2}+2\beta}}{(1+\nu+\sqrt{1-\nu^2-\upsilon^2}\lambda)^\frac12}\left(1+\frac{\nu}{\sqrt{1-\nu^2-\upsilon^2}}\lambda\right)\,\d\mu(\lambda)\\
&\qquad\qquad
-
\frac12
\int_0^1
\frac{(1-\nu+\sqrt{1-\nu^2-\upsilon^2}\lambda)^{\frac{d-2}{2}+2\beta}}{(1+\nu+\sqrt{1-\nu^2-\upsilon^2}\lambda)^\frac32}\left(1-\frac{\nu}{\sqrt{1-\nu^2-\upsilon^2}}\lambda\right)\,\d\mu(\lambda),
\end{align*}
which is non-positive since $\beta\geq\frac{2-d}{4}$ and 
\[
1-\frac{\nu}{\sqrt{1-\nu^2-\upsilon^2}}\lambda\geq0
\]
for $\nu\in\big[0,\sqrt{\tfrac{1-\upsilon^2}{2}}\big]$.  On the other hand, for $\nu\in\big[\sqrt{\tfrac{1-\upsilon^2}{2}},\sqrt{1-\upsilon^2}\big]$, which imposes $0\leq\sqrt{1-\nu^2-\upsilon^2}\leq\nu$, it follows that 
\begin{align*}
\Xi(\nu,\upsilon)
&=
\int_0^1
\frac{
(
1-\nu-\sqrt{1-\nu^2-\upsilon^2}\lambda
)^{\frac{d-2}{2}+2\beta}
}{
(
1+\nu-\sqrt{1-\nu^2-\upsilon^2}\lambda
)^\frac12
}
\,\d\mu(\lambda)
+
\int_0^1
\frac{
(
1-\nu+\sqrt{1-\nu^2-\upsilon^2}\lambda
)^{\frac{d-2}{2}+2\beta}
}{
(
1+\nu+\sqrt{1-\nu^2-\upsilon^2}\lambda
)^\frac12
}
\,\d\mu(\lambda)\\
&\leq
\int_0^1 2 \,\d\mu(\lambda)\\
&
\leq
\int_0^1(1-\sqrt{1-\upsilon^2}\lambda)^{\frac{d-3}{2}+2\beta}\,\d\mu(\lambda)
+
\int_0^1(1+\sqrt{1-\upsilon^2}\lambda)^{\frac{d-3}{2}+2\beta}\,\d\mu(\lambda)\\
&=
\Xi(0,\upsilon).
\end{align*}
Here, the first inequality is justified as long as $\beta\geq\frac{2-d}{4}$ and  the second inequality is given by the arithmetic-geometric mean:
\begin{align*}
\frac12
\left(
(1-\sqrt{1-\upsilon^2}\lambda)^{\frac{d-3}{2}+2\beta}
+
(1+\sqrt{1-\upsilon^2}\lambda)^{\frac{d-3}{2}+2\beta}
\right)
\geq
\left(
1-(1-\upsilon^2)\lambda^2
\right)^{\frac{d-3}{4}+\beta}
\geq1.
\end{align*}
Since the second inequality of \eqref{i:Xi Xi Xi(0,0)} can be readily proved by Lemma \ref{lem:monotonicity}, we have \eqref{ineq:biest D_-D_+ } with $C=\F(\beta,d)^\frac12$ for $\beta\in[\frac{2-d}{4},\frac{3-d}{4}]$ as well.

\end{proof}

\subsection{Threshold of our argument for $\beta\in(\frac{3-d}{4},\frac{5-d}{4})$}\label{sec:arg. on gap}
Although $C=\F(\beta,d)^\frac12$ will be shown to be optimal for $\beta\in[\frac{2-d}{4},\frac{3-d}{4}]\cup[\frac{5-d}{4},\infty)$ in the case of radial data, it remains unclear whether this continues to be true for $\beta\in(\frac{3-d}{4},\frac{5-d}{4})$; here we establish that there is no way to obtain the constant $\F(\beta,d)^\frac12$ if one first makes use of Theorem \ref{thm:main}.
In order to show that, we shall invoke the following useful result for the beta function due to Agarwal--Barnett--Dragmir \cite{ABD00}:

\begin{lemma}[\cite{ABD00}]\label{lem:B<B}
Let $m$, $p$ and $k\in\mathbb R$ satisfy $m$, $p>0$, and $p>k>-m$. If we have 
\begin{equation}\label{cond:B<B}
k(p-m-k)>0
\end{equation}
then
\[
B(p,m)>B(p-k,m+k)
\]
holds.
\end{lemma}
\begin{figure}[b]
\begin{center}
\begin{tikzpicture}[scale=1]

  \draw [color=blue, fill=blue!30!,opacity=0.5, thick] (3,0.3)--(3.4,0.3)--(3.4,0.5)--(3,0.5)--(3,0.3);
  
  \draw [->](-0.5,0)--(4,0);
  \draw [->](0,-0.5)--(0,4);
  
  \draw [color=red](0.4,4) .. controls (0.4,0.4) ..(4,0.4);
  

  \draw [color=blue,dotted](3,0)--(3,4);
  \draw [color=blue,dotted](3.4,0)--(3.4,4);  
  \draw [color=blue,dotted](0,0.5)--(4,0.5);
  \draw [color=blue,dotted](0,0.3)--(4,0.3);

  \node at (-0.4,-0.4) {$O$};
  \node [right] at (4,0) {$r_1$};
  \node [left] at (0,4) {$r_2$};
  \node [below] at (3.2,0) {$\sim\frac 1\delta$};
  \node [left] at (0,0.4) {$\sim\delta$};
  \node [left] at  (3,0.8) {$\textcolor{blue}{\mathcal O_\delta}$};
   
\end{tikzpicture}
\caption{The set $\mathcal O_\delta$ along the curve $r_1=r_2^{-1}$.} \label{fig:counter example}
\end{center}
\end{figure}
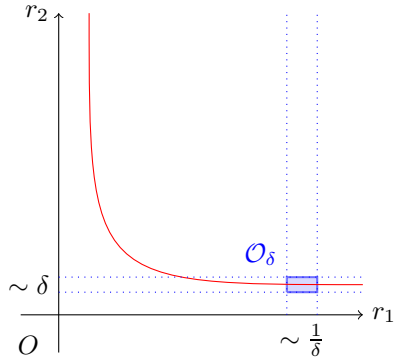

With this in hand, we establish the following.
\begin{proposition}\label{prop:a counter example}
Let $d\ge2$ and $\beta\in(\frac{3-d}{4},\frac{5-d}{4})$. Then there exist radially symmetric $f$ and $g$ such that
\begin{align*}
&\KG(\beta,d)
\int_{\mathbb{R}^{2d}}
|\widehat{f}(\eta_1)|^2|\widehat{g}(\eta_2)|^2\phi_s(|\eta_1|)\phi_s(|\eta_2|)
\K_\frac12^{\frac{d-2}{2}+2\beta}(\eta_1,\eta_2)
\,\d\eta_1\d\eta_2\\
&\quad>
\F(\beta,d)
\|\phi_s(D)^{\frac{d-1}{4}+\beta}f\|_{L^2(\mathbb{R}^d)}^2
\|\phi_s(D)^{\frac {d-1}{4}+\beta}g\|_{L^2(\mathbb{R}^d)}^2
\end{align*}
holds.
\end{proposition}

\begin{proof}
Let $0<\delta\ll1$ and 
\[
A=\left\{\xi\in\mathbb{R}^d\,:\, \frac12<|\xi|<2\right\}.
\]
Define $f_A$ and $g_A$ so that for $\xi\in\mathbb{R}^d$ 
\[
\widehat{f_A}(\xi)=\chi_A(\tfrac\xi\delta)\quad\text{and}\quad\widehat{g_A}(\xi)=\chi_A(\delta\xi),
\]
where $\chi_A$ is the characteristic function of $A$. By use of polar coordinates,
\begin{align*}
&\int_{\mathbb{R}^{2d}}
|\widehat{f_A}(\eta_1)|^2|\widehat{g_A}(\eta_2)|^2\phi_s(|\eta_1|)\phi_s(|\eta_2|)
\K_\frac12^{\frac{d-2}{2}+2\beta}(\eta_1,\eta_2)
\,\d\eta_1\d\eta_2\\
&\qquad=
\int_{\mathcal{O}_\delta}|\widehat{f_A}(r_1)|^2|\widehat{g_A}(r_2)|^2(\phi_s(r_1)\phi_s(r_2))^{\frac{d-1}{2}+2\beta}\Theta_\frac12^{\frac{d-2}{2}+2\beta}(r_1,r_2)^{d-1}\,\d r_1\d r_2.
\end{align*}
Here, the set $\mathcal O_\delta$ is defined by (see also Figure \ref{fig:counter example})
\[
\mathcal{O}_\delta
=
\left\{
(r_1,r_2):\frac{1}{2\delta}<r_1<\frac{2}{\delta},\frac{\delta}{2}<r_2<2\delta
\right\}.
\]
Now, for $(r_1,r_2)\in\mathcal{O}_\delta$, by taking the limit $\delta\to0$ represented by $\phi_s(r_1)\to \infty$ and $\phi_s(r_2)\to s$, it follows that  
\begin{align*}
\Theta_\frac12^{\frac{d-2}{2}+2\beta}(r_1,r_2)
\to
|\mathbb S^{d-1}|^2.
\end{align*}
Therefore, for sufficiently small $\delta>0$, 
\begin{align*}
&\KG(\beta,d)
\int_{\mathbb{R}^{2d}}
|\widehat{f_A}(\eta_1)|^2|\widehat{g_A}(\eta_2)|^2\phi_s(|\eta_1|)\phi_s(|\eta_2|)
\K_\frac12^{\frac{d-2}{2}+2\beta}(\eta_1,\eta_2)
\,\d\eta_1\d\eta_2\\
&\quad=
(2\pi)^{2d}
\KG(\beta,d)
\|\phi_s(\sqrt{-\Delta})^{\frac{d-1}{4}+\beta}f_A\|_{L^2(\mathbb{R}^d)}^2
\|\phi_s(\sqrt{-\Delta})^{\frac {d-1}{4}+\beta}g_A\|_{L^2(\mathbb{R}^d)}^2,
\end{align*}
and it is enough to show
\begin{align}\label{ineq:counterexample}
(2\pi)^{2d}\KG(\beta,d)
>
\F(\beta,d)
\end{align}
for $d\geq2$ and $\beta\in(\frac{3-d}{4},\frac{5-d}{4})$. By the formula \eqref{eq:volume of d-sphere} and the definitions of constants, this can be simplified as
\[
B(\tfrac{3d-5}{4}+\beta,\tfrac{3d-3}{4}+\beta)>B(d-2+2\beta,\tfrac d2)
\]
which, if fact, follows from Lemma \ref{lem:B<B} by letting $p=\frac{3d-5}{4}+\beta$, $m=\frac{3d-3}{4}+\beta$ and $k=\frac{3-d}{4}-\beta$ for $d\geq2$ and $\beta\in (\frac{3-d}{4},\frac{5-d}{4})$. Note that for the specific triple $(p,m,k)$, the hypothesis \eqref{cond:B<B} is equivalent to considering $\beta$ from the gap $(\frac{3-d}{4},\frac{5-d}{4})$ when $d\geq2$.
\end{proof}

\subsection{Contributions of radial symmetry}
Here, we observe for general (not necessary radially symmetric) data $f$ and $g$ the inequality  \eqref{ineq:biest D_-D_+ } holds only if $\beta\geq\frac{3-d}{4}$, in other words, the radial symmetry condition on $f$ and $g$ widens the range of the regularity parameter $\beta$. The proof is based on the Knapp type argument in \cite{FK00} where they proved $\beta_-\geq\frac{3-d}{4}$ is necessary for \eqref{ineq:null +-} to hold.

\begin{proposition}\label{prop:necessary condition}
Let $\beta<\frac{3-d}{4}$. For any $C_*>0$, there exists $f,g\in H^{\frac{d-1}{4}+\beta}(\mathbb{R}^d)$ such that
\begin{align}
&\||\square|^{\beta}(e^{it\phi_s(\sqrt{-\Delta})}f\overline{e^{it\phi_s(\sqrt{-\Delta})}g})\|_{L^2(\mathbb{R}^{d+1})}^2\\
&\quad> 
C_*
\|\phi_s(\sqrt{-\Delta})^{\frac{d-1}{4}+\beta}f\|_{L^2(\mathbb{R}^d)}^2
\|\phi_s(\sqrt{-\Delta})^{\frac {d-1}{4}+\beta}g\|_{L^2(\mathbb{R}^d)}^2.\nonumber
\end{align}
\end{proposition}

\begin{proof}
For $\eta_1\in\mathbb{R}^d$ (similarly, for $\eta_2\in\mathbb R^d$), we set indices $(1), \ldots, (d)$ to indicate  components of vectors, namely, $\eta_1=(\eta_{1(1)},\ldots,\eta_{1(d)})$. Also, denote $\eta'_1=(\eta_{1(2)},\ldots,\eta_{1(d)})\in\mathbb{R}^{d-1}$ and $\eta''_1=(\eta_{1(3)},\ldots,\eta_{1(d)})\in\mathbb{R}^{d-2}$. Now, for large $L>0$, eventually sent to infinity, define sets $\setF$ and $\setG$ by
\[
\setF=\{\eta\in\mathbb{R}^d:L\le\eta_{(1)}\le 2L, 1\le\eta_{(2)}\le2,|\eta''|\le1\}
\]
and 
\[
\setG=\{\eta\in\mathbb{R}^d:L\le\eta_{(1)}\le 2L, -1\le\eta_{(2)}\le-2,|\eta''|\le1\}.
\]
For such $f$ and $g$ 
\[
\left|
|\square|^\beta
(e^{it\phi_s(D)}f(x)\overline{e^{it\phi_s(D)}g(x)})
\right|
\sim
\left|
\int_\setF\int_\setG e^{i\Phi_s(x,t:\eta_1,\eta_2)}\K_{-\beta}^0(\eta_1,\eta_2)\,\d\eta_1\d\eta_2
\right|,
\]
where
\[
\Phi_s(x,t:\eta_1,\eta_2)=x\cdot(\eta_1-\eta_2)+t(\phi_s(|\eta_1|)-\phi_s(|\eta_2|)).
\]
Now, we follow the idea of Knapp's example to derive a lower bound. From the setting
we have $|\eta_1|\sim|\eta_2|\sim\phi_s(|\eta_1|)\sim\phi_s(|\eta_2|)\sim|\eta_1+\eta_2|\sim L$, $|\eta_{1(1)}-\eta_{2(1)}|\sim\theta\sim L^{-1}$, $|\phi_s(|\eta_1|)-\eta_{1(1)}|\sim|\eta_1'|^2|\eta_1|^{-1}\sim |\phi_s(|\eta_2|)-\eta_{2(1)}|\sim|\eta_2'|^2|\eta_2|^{-1}\sim L^{-1}$,  and $|\eta_1'+\eta_2'|\lesssim1$ for $(\eta_1,\eta_2)\in \setF\times \setG$. Then, it follows that 
\begin{align}\label{ineq:W beta pre-estimate}
(\phi_s(|\eta_1|)\phi_s(|\eta_2|))^2-(\eta_1\cdot\eta_2-s^2)^2
\sim
s^2|\eta_1+\eta_2|^2+|\eta_1|^2|\eta_2|^2\sin^2\theta
\sim
L^2
\end{align} 
and hence
\[
\K_{-\beta}^0(\eta_1,\eta_2)
\sim
\left(
\frac{(\phi_s(|\eta_1|)\phi_s(|\eta_2|))^2-(\eta_1\cdot\eta_2-s^2)^2}{\phi_s(|\eta_1|)\phi_s(|\eta_2|)+\eta_1\cdot\eta_2+s^2}
\right)^\beta
\sim
1.
\]
Moreover, for the phase, then it follows that
\begin{align*}
&|\Phi_s(x,t:\eta_1,\eta_2)|\\
&\quad=
|t(\phi_s(|\eta_1|)-\eta_{1(1)}-\phi_s(|\eta_2|)+\eta_{2(1)})+(x_{(1)}+t)(\eta_{1(1)}-\eta_{2(1)})+x'\cdot(\eta_1'+\eta_2')|\\
&\quad\le
 |t|L^{-1}+|x_{(1)}+t|L^{-1}+|x'|<\frac\pi3
\end{align*}
for $(x,t)=(x_{(1)},x',t)$ in a slab $\setR=[-L^{-1},L^{-1}]\times[-1,1]^{d-1}\times[-L,L]$  whose volume is the order of $1$. Hence, 
\[
|\square|^\beta
(e^{it\phi_s(D)}f(x)\overline{e^{it\phi_s(D)}g(x)})
\gtrsim
|\setF||\setG|\chi_\setR(x,t)
\]
and so
\[
\||\square|^{\beta}
(e^{it\phi_s(D)}f(x)\overline{e^{it\phi_s(D)}g(x)})
\|_{L^2(\mathbb R^{d+1})}^2
\gtrsim
|\setF|^2|\setG|^2|\setR|\sim|\setF|^2|\setG|^2.
\]

%

On the other hand, we have
\[
\|\phi_s(\sqrt{-\Delta})^{\frac{d-1}{4}+\beta}f\|_{L^2(\mathbb{R}^d)}^2
\|\phi_s(\sqrt{-\Delta})^{\frac {d-1}{4}+\beta}g\|_{L^2(\mathbb{R}^d)}^2\lesssim L^{d-1+4\beta}|\setF||\setG|.
\]
Therefore, it is implied that 
\[
|\setF|^2|\setG|^2\lesssim L^{d-1+4\beta}|\setF||\setG|.
\]
The fact $|\setF|\sim|\setG|\sim L$ and letting $L\to\infty$ result in the desired necessary condition
\[
\frac{3-d}{4}\le\beta.
\]
\end{proof}

\begin{figure}[t]
\begin{center}
\begin{tikzpicture}[rotate around x=-60, rotate around y=0,rotate around z=-10,scale=0.45]
\fill [fill=lightgray] (10,1,0)--(10,2,0)--(18,2,0)--(18,1,0);
\fill [fill=lightgray] (10,-1,0)--(10,-2,0)--(18,-2,0)--(18,-1,0);
\fill [fill=red!30!,opacity=0.5] (10,1,1)--(10,2,1)--(18,2,1)--(18,2,-1)--(18,1,-1)--(10,1,-1);
\fill [fill=blue!30!,opacity=0.5] (10,-1,1)--(10,-2,1)--(10,-2,-1)--(18,-2,-1)--(18,-1,-1)--(18,-1,1);

\draw [->] (-1,0,0)--(22,0,0); \node [right] at (22,0,0) {$\eta_{(1)}$};
\draw [->] (0,-4,0)--(0,4,0); \node [above] at (0,4,0) {$\eta_{(2)}$};
\draw [->] (0,0,-3)--(0,0,3); \node [above] at (0,0,3) {$\eta''$};

\draw [dotted] (0,2,0)--(18,2,0);
\draw [dotted] (0,1,0)--(18,1,0);
\draw [dotted] (10,0,0)--(10,2,0);
\draw [dotted] (18,0,0)--(18,2,0)--(0,2,0);

\draw [dotted] (0,-2,0)--(18,-2,0);
\draw [dotted] (0,-1,0)--(18,-1,0);
\draw [dotted] (10,0,0)--(10,-2,0);
\draw [dotted] (18,0,0)--(18,-2,0)--(0,-2,0);

\draw  (10,1,-1)--(10,2,-1)--(18,2,-1)--(18,1,-1)--(10,1,-1);
\draw  (10,1,1)--(10,2,1)--(18,2,1)--(18,1,1)--(10,1,1);
\draw  (10,1,-1)--(10,2,-1)--(10,2,1)--(10,1,1)--(10,1,-1);
\draw  (18,1,-1)--(18,2,-1)--(18,2,1)--(18,1,1)--(18,1,-1);

\draw [thick] (18,1,-1)--(18,2,-1)--(18,2,1)--(18,1,1)--(18,1,-1);
\draw [thick] (18,1,-1)--(10,1,-1)--(10,1,1)--(18,1,1);
\draw [thick] (10,1,1)--(10,2,1)--(18,2,1);

\draw  (10,-1,-1)--(10,-2,-1)--(18,-2,-1)--(18,-1,-1)--(10,-1,-1);
\draw  (10,-1,1)--(10,-2,1)--(18,-2,1)--(18,-1,1)--(10,-1,1);
\draw  (10,-1,-1)--(10,-2,-1)--(10,-2,1)--(10,-1,1)--(10,-1,-1);
\draw  (18,-1,-1)--(18,-2,-1)--(18,-2,1)--(18,-1,1)--(18,-1,-1);

\draw [thick] (18,-1,-1)--(18,-2,-1)--(18,-2,1)--(18,-1,1)--(18,-1,-1);
\draw [thick] (18,-2,1)--(10,-2,1)--(10,-1,1)--(18,-1,1)--(18,-2,1);
\draw [thick] (10,-2,1)--(10,-2,-1)--(18,-2,-1);

\draw (0,0,0)--(16,1.5,-0.5);
\fill (16,1.5,-0.5) circle (2pt);
\draw (0,0,0)--(12,-1.3,0.5);
\fill (12,-1.3,0.5) circle (2pt);
\fill [fill=white] (5.6,0.1,0)--(6.4,0.1,0)--(6.4,-0.1,0)--(5.6,-0.1,0);
\draw (5,-0.6,0.3) to[out=20, in=-70] (5,0.35,0);
\node   at (6,0,0) {$\theta$};  

\node at (12,3,1) {$\textcolor{red}{\setF}$};
\node at (12,-3,-1) {$\textcolor{blue}{\setG}$};
\node[below] at (19,0,0) {$2L$};
\node [left] at (0,1,0) {$1$};
\node [left] at (0,2,0) {$2$};
\node [left] at (-0.2,0,-0.5) {$O$};
\end{tikzpicture}
\caption{The sets $\setF$ and $\setG$, which are sent away from the origin along $\eta_{(1)}$-axis.} \label{fig:setF and setG}
\end{center}
\end{figure}

\section{Sharpness of constants}\label{sec:sharpness of constants}
Let us begin with some supplemental discussions on the refined Strichartz estimates \eqref{ineq:Strichartz refined general J} and \eqref{ineq:Strichartz refined non-wave}. Then,  we focus on completing our proof of Corollaries \ref{cor:biest radial}, \ref{cor:wave} and \ref{cor:non-wave} by proving that the stated constants are optimal and non-existence of extremisers. We achieve optimality of constants by considering the functions $f_a$ given by \eqref{extremisers}; this is a natural guess given that such functions are extremisers for \eqref{ineq:main}, as shown in our proof of Theorem \ref{thm:main}. Before proceeding, we introduce the following useful notation.
\[
\L_a(\beta)
:=
\int_{4as}^\infty
e^{-\rho}
\int_0^{(2a)^{-1}\sqrt{\rho^2-(4as)^2}}
\frac{(\rho^2-(2ar)^2-(4as)^2)^{\frac{d-2}{2}+2\beta}}{\rho^2-(2ar)^2}r^{d-1}
\,\d r\d\rho
\]
and 
\[
\R_a(\beta,b(\beta))
:=
\left(
\int_{2as}^\infty e^{-\rho}\rho^{b(\beta)}(\rho^2-(2as)^2)^{\frac{d-2}{2}}\,\d\rho
\right)^2.
\]

\subsection{On the refined Strichartz estimates}

It is straightforward that the estimates \eqref{ineq:corollaries} with claimed constants in Corollaries \ref{cor:wave} and \ref{cor:non-wave} when $(\alpha,\beta)=(\frac12,\frac{3-d}{4})$ and $(\alpha,\beta)=(\frac12,\frac{2-d}{4})$ coincide with the results obtained by applying the kernel estimates \eqref{ineq:kernelest wave} and \eqref{ineq:kernelest non-wave} to \eqref{ineq:main}, respectively. To obtain the estimates \eqref{ineq:Strichartz refined general J} and \eqref{ineq:Strichartz refined non-wave}, we require the additional fact that 
\[
\int_{\mathbb R^{2d}}f(x)f(y)x\cdot y\,\d x\d y\geq0.
\]
Indeed, in the wave regime, after we apply the kernel estimate \eqref{ineq:kernelest wave} to \eqref{ineq:main}, it follows that
\begin{align*}
&\int_{\mathbb R^{2d}}
|\widehat{f}(\eta_1)|^2|\widehat{f}(\eta_2)|^2
\phi_s(|\eta_1|)\phi_s(|\eta_2|)
\K_0^1(\eta_1,\eta_2)
\,\d\eta_1\d\eta_2\\
&\quad\leq
\int_{\mathbb R^{2d}}
|\widehat{f}(\eta_1)|^2|\widehat{f}(\eta_2)|^2
\phi_s(|\eta_1|)\phi_s(|\eta_2|)
(\phi_s(|\eta_1|)\phi_s(|\eta_2|)-s^2)
\,\d\eta_1\d\eta_2,
\end{align*}
which immediately yields \eqref{ineq:Strichartz refined general J}. Similarly, one can deduce \eqref{ineq:Strichartz refined non-wave} in the non-wave regime. Finally, the estimate \eqref{ineq:corollaries} with $C=\F(\frac{5-d}{4},d)^\frac12$ when $(\alpha,\beta)=(1,\frac{5-d}{4})$ is obtained by further estimating the kernel of \eqref{ineq:Strichartz refined general J} as
\[
\phi_s(|\eta_1|)\phi_s(|\eta_2|)-s^2
\leq
\phi_s(|\eta_1|)\phi_s(|\eta_2|).
\]
Again, we will see the sharpness of constants below. Of course, by a similar argument to the above, one can easily obtain the estimate \eqref{ineq:corollaries} with 
\begin{equation}\label{const:non-wave}
C=\frac{2^{-d+1}\pi^{\frac{-d+2}{2}}}{s\Gamma(\frac{d+2}{2})}
\end{equation}
when $(\alpha,\beta)=(1,\frac{4-d}{4})$ from \eqref{ineq:Strichartz refined non-wave} in the non-wave regime, and it is natural to hope that the constant is still optimal. We do not, however, know whether or not  the constant \eqref{const:non-wave} is optimal, which will become clear from the following argument on the sharpness of constants.\\


\subsection{Wave regime}
Recall $\beta_d=\max\{\frac{1-d}{4},\frac{2-d}{2}\}$. We shall consider \eqref{ineq:corollaries} with $(\alpha,\beta)=(\frac{d-1}{4}+\beta,\beta)$ for $\beta\in(\beta_d,\infty)$. 
For $f_a$ given by \eqref{extremisers}, one can observe that 
\[
\||\square|^\beta|e^{it\phi_s(D)}f_a|^2\|_{L^2(\mathbb R^{d+1})}^2
=
2^{\frac{-3d+7}{2}-2\beta}|\mathbb S^{d-1}|\KG(\beta,d)(2a)^{-2d+5-4\beta}
\L_a(\beta)
\]
and
\begin{equation}\label{Strichartz RHS D}
\|\phi_s(D)^{\frac{d-1}{4}+\beta}f_a\|_{L^2(\mathbb R^d)}^4
=
(2\pi)^{-2d}|\mathbb S^{d-1}|^2(2a)^{-3d+5-4\beta}\R_a(\beta,\tfrac{d-3}{2}+2\beta),
\end{equation}
and so it is enough to show
\begin{equation}\label{eq:LHS/RHStoF}
\lim_{a\to0}
\frac{\||\square|^\beta|e^{it\phi_s(D)}f_a|^2\|_{L^2(\mathbb R^{d+1})}^2}{\|\phi_s(D)^{\frac{d-1}{4}+\beta}f_a\|_{L^2(\mathbb R^d)}^4}
=
\lim_{a\to0}(2a)^d C(\beta,d)\frac{\L_a(\beta)}{\R_a(\beta,\tfrac{d-3}{2}+2\beta)}
=
\F(\beta,d),
\end{equation}
where
\[
C(\beta,d)
=
2^{-2(d-2)}\pi^{\frac{-d+1}{2}}\frac{\Gamma(\frac{d-1}{2}+2\beta)}{\Gamma(d-1+2\beta)}.
\]
Since we have, by appropriate change of variables,
\[
\L_a(\beta)
=
e^{-4as}
(2a)^{-d}
\int_0^\infty
 e^{-\rho}\rho^{\frac32d-2+2\beta}(\rho+8as)^{\frac32d-2+2\beta}
\int_0^1
\frac{(1-\nu^2)^{d-2+2\beta}\nu^{d-1}}{(\rho+4as)^2(1-\nu^2)+(4as)^2\nu^2}
\,\d\nu\d\rho
\]
and
\[
\R_a(\beta,\tfrac{d-3}{2}+2\beta)
=
e^{-4as}
\left(
\int_0^\infty 
e^{-\rho}
(\rho+2as)^{\frac{d-3}{2}+2\beta}
\rho^{\frac{d-2}{2}}(\rho+4as)^{\frac{d-2}{2}} 
\,\d\rho
\right)^2,
\]
one may deduce
\[
\lim_{a\to0}
(2a)^d\frac{\L_a(\beta)}{\R_a(\beta,\tfrac{d-3}{2}+2\beta)}
=
\frac{\Gamma(3d-5+4\beta)B(d-2+2\beta,\frac d2))}{2\Gamma(\frac{3d-5}{2}+2\beta)^2},
\]
which leads to \eqref{eq:LHS/RHStoF}.

In order to show the constant $\F(\frac{5-d}{4},d)^\frac12$ is sharp in \eqref{ineq:Strichartz refined general J}, we apply a similar calculation. In particular, one may note that the right-hand side of \eqref{ineq:Strichartz refined general J} can be written as
\begin{equation}
(2\pi)^{-2d}|\mathbb S^{d-1}|^2(2a)^{-2d}[\R_a(\tfrac{5-d}{4},1)-(2as)^2\R_a(\tfrac{5-d}{4},0)],
\end{equation}
instead of \eqref{Strichartz RHS D}. One can also see the second term is negligible in the sense of the optimal constant since it vanishes while $a$ tends to $0$. 
\subsection{Non-wave regime}\label{s:sharpness non-wave}
Let $f_a$ satisfy \eqref{extremisers}. Note that in the non-wave regime the right-hand side of \eqref{ineq:corollaries} for the pair $(\frac d4+\beta,\beta)$ is expressed as
\begin{equation}
\|\phi_s(D)^{\frac{d}{4}+\beta}f_a\|_{L^2(\mathbb R^d)}^4
=
(2\pi)^{-2d}|\mathbb S^{d-1}|^2(2a)^{-3d+4-4\beta}\R_a(\beta,\tfrac{d-2}{2}+2\beta).
\end{equation}
Then, as we have done above, reform $\L_a(\beta)$ and $\R_a(\beta,\tfrac{d-2}{2}+2\beta)$ as follows by some appropriate change of variables:
\[
\L_a(\beta)
=
e^{-4as}
(2a)^{\frac d2-4+2\beta}
\int_0^\infty
 e^{-\rho}\rho^{\frac32d-2+2\beta}(\frac{\rho}{2a}+4s)^{\frac32d-2+2\beta}
\int_0^1
\frac{(1-\nu^2)^{d-2+2\beta}\nu^{d-1}}{(\frac{\rho}{2a}+2s)^2(1-\nu^2)+(2s)^2\nu^2}
\,\d\nu\d\rho
\]
and
\[
\R_a(\beta,\tfrac{d-2}{2}+2\beta)
=
e^{-4as}
(2a)^{2d-4+4\beta}
\left(
\int_0^\infty 
e^{-\rho}
(\frac{\rho}{2a}+s)^{\frac{d-2}{2}+2\beta}
\rho^{\frac{d-2}{2}}(\frac{\rho}{2a}+2s)^{\frac{d-2}{2}} 
\,\d\rho
\right)^2.
\]
First, we shall consider \eqref{ineq:corollaries} with $(\alpha,\beta)=(0, \frac{2-d}{4})$. 
By a similar argument to the wave regime above, one can easily check that
\[
\lim_{a\to\infty}(2a)^{d+1}\frac{\L_a(\tfrac{2-d}{4})}{\R_a(\tfrac{2-d}{4},0)}=2^{d-3}s^{-1}
\]
holds, from which it follows that
\[
\lim_{a\to\infty}
\frac{\||\square|^{\frac{2-d}{4}}|e^{it\phi_s(D)}f_a|^2\|_{L^2(\mathbb R^{d+1})}^2}{\|\phi_s(D)^{\frac12}f_a\|_{L^2(\mathbb R^d)}^4}
=
\lim_{a\to\infty}(2a)^{d+1} C(\tfrac{2-d}{4},d)\frac{\L_a(\tfrac{2-d}{4})}{\R_a(\tfrac{2-d}{4},0)}
=
\frac{2^{-d+1}\pi^{\frac{-d+2}{2}}}{s\Gamma(\frac d2)}.
\]\\

We now turn to \eqref{ineq:Strichartz refined non-wave} ,where $(\alpha,\beta)=(1,\frac{4-d}{4})$, and the argument goes almost the same as above. In this case, we have 
$$C(\tfrac{4-d}{4},d)
=
\frac{2^{-d+2}\pi^{\frac{-d+2}{2}}}{\Gamma(\frac{d+2}{2})},$$
\[
\L_a(\tfrac{4-d}{4})
=
e^{-4a}(2a)^{-2}
\int_0^\infty e^{-\rho}\rho^d\left(\frac{\rho}{2a}+4s\right)^d\left(\int_0^\infty\frac{(1-\nu^2)^\frac s2\nu^{d-1}}{(\frac{\rho}{2a}+2s)^2(1-\nu^2)+(2s)^2\nu^2}\,\d\nu\right)\,\d\rho
\]
and 
\begin{align*}
&
((2a)^{d-2}e^{-4as})^{-1}
\left(
\R_a(\tfrac{4-d}{4},1)
-
(2as)^2
\R_a(\tfrac{4-d}{4},0)
\right)\\
&\quad=
\left(
\int_0^\infty e^{-\rho}(\rho+2as)\rho^{\frac{d-2}{2}}\left(\frac{\rho}{2a}+2s\right)^{\frac{d-2}{2}}\,\d\rho
\right)^2
-
\left(
(2as)
\int_0^\infty e^{-\rho}\rho^{\frac{d-2}{2}}\left(\frac{\rho}{2a}+2s\right)^{\frac{d-2}{2}}\,\d\rho
\right)^2\\
&\quad=
(2a)
\left(
\int_0^\infty e^{-\rho}\rho^{\frac{d-2}{2}}\left(\frac{\rho}{2a}+2s\right)^{\frac{d}{2}}\,\d\rho
\right)
\left(
\int_0^\infty e^{-\rho}\rho^{\frac{d}{2}}\left(\frac{\rho}{2a}+2s\right)^{\frac{d-2}{2}}\,\d\rho
\right).
\end{align*}
Hence, one can easily check that 
\begin{align}\label{eq:LHS/RHStoJ}
&\lim_{a\to\infty}
\frac{\||\square|^{\frac{4-d}{4}}|e^{it\phi_s(D)}f_a|^2\|_{L^2(\mathbb R^{d+1})}^2}{\|\phi_s(D)f_a\|_{L^2(\mathbb R^d)}^4-(2as)^2\|\phi_s(D)^\frac12f_a\|_{L^2(\mathbb R^d)}^4}\nonumber\\
&\quad=
\lim_{a\to\infty}(2a)^{d+1} C(\tfrac{4-d}{4},d)\frac{\L_a(\tfrac{4-d}{4})}{\R_a(\tfrac{4-d}{4},1)-(2as)^2\R_a(\tfrac{4-d}{4},0)}\nonumber\\
&\quad=
\frac{2^{-d+2}\pi^{\frac{-d+2}{2}}\Gamma(d+1)\left(\int_0^\infty(1-\nu^2)^\frac d2\nu^{d-1}\,\d\nu\right)}{s\Gamma(\frac{d+2}{2})\Gamma(\frac d2)\Gamma(\frac d2+1)}\\
&\quad=
\frac{2^{-d+1}\pi^{\frac{-d+2}{2}}}{s\Gamma(\frac {d+2}{2})}\nonumber
\end{align}
by noticing $\int_0^\infty(1-\nu^2)^\frac d2\nu^{d-1}\,\d\nu=\frac12B(\frac d2+1,\frac d2)$.


In contrast to the wave regime, the squared right-hand side of \eqref{ineq:Strichartz refined general J} without the constant can be expressed as 
\[
(2\pi)^{-2d}|\mathbb S^{d-1}|^2(2a)^{-2d}(\R_a(\tfrac{4-d}{4},1)-(2as)^2\R_a(\tfrac{4-d}{4},0)).
\]
Unlike the wave regime, however, $a$ is sent to $\infty$ (instead of $0$) to derive \eqref{eq:LHS/RHStoJ} and the second term of \eqref{ineq:Strichartz refined general J} does not vanish.  Hence, we cannot follow the argument for the wave regime and do not know whether the constant \eqref{const:non-wave} is still optimal for \eqref{ineq:corollaries} when $(\alpha,\beta)=(1,\frac{4-d}{4})$. 


\subsection{Non-existence of an extremiser}
Suppose there were non-trivial $f$ and $g$ that satisfy any of the statements in Corollary \ref{cor:non-wave} with equality.
From our proof via Theorem \ref{thm:main}, it would be required that
\begin{align*}
&\int_{\mathbb{R}^{2d}}
|\widehat{f}(\eta_1)|^2|\widehat{g}(\eta_2)|^2\phi_s(|\eta_1|)\phi_s(|\eta_2|)
\K_\frac12^{\frac{d-2}{2}+2\beta}(\eta_1,\eta_2)
\,\d \eta_1\d \eta_2\\
&\quad=
2^{-\frac12}s^{-1}
\int_{\mathbb{R}^{2d}}
|\widehat{f}(\eta_1)|^2|\widehat{g}(\eta_2)|^2\phi_s(|\eta_1|)\phi_s(|\eta_2|)
\K_0^{\frac{d-2}{2}+2\beta}(\eta_1,\eta_2)
\,\d \eta_1\d \eta_2
\end{align*}
holds. Then, 
\[
\int_{\mathbb{R}^{2d}}
|\widehat{f}(\eta_1)|^2|\widehat{g}(\eta_2)|^2\phi_s(|\eta_1|)\phi_s(|\eta_2|)
\left(
\K_\frac12^{\frac{d-2}{2}+2\beta}(\eta_1,\eta_2)
-
2^{-\frac12}s^{-1}
\K_0^{\frac{d-2}{2}+2\beta}(\eta_1,\eta_2)
\right)
\,\d \eta_1\d \eta_2
=0
\]
would hold. Since $f$, $g$ are assumed to be non-trivial $\widehat{f}$, $\widehat{g}\not=0$ on some set $\setF\times \setG\subseteq\mathbb R^{2d}$ with $|\setF|$, $|\setG|>0$, it would be deduced that
\begin{equation}\label{eq:kernerls}
\K_\frac12^{\frac{d-2}{2}+2\beta}(\eta_1,\eta_2)
-
2^{-\frac12}s^{-1}
\K_0^{\frac{d-2}{2}+2\beta}(\eta_1,\eta_2)
=
0
\end{equation}
on $(\setF\times \setG)\setminus \setN$ where $\setN\subseteq\mathbb R^{2d}$ is a null set. However, \eqref{eq:kernerls} would hold only on the diagonal line $\{(\eta_1,\eta_2):\eta_1=\eta_2\}$ (the equality condition of \eqref{ineq:kernelest non-wave}), which is a null set and so is $\{(\eta_1,\eta_2):\eta_1=\eta_2\}\cap(\setF\times \setG)$. This is a contradiction. \\

For Corollary \ref{cor:biest radial}, Corollary \ref{cor:wave}, similar arguments above can be carried. In particular, for equality in the wave regime, the formula \eqref{eq:kernerls} might be replaced by
\[
\K_\frac12^{\frac{d-2}{2}+2\beta}(\eta_1,\eta_2)
-
\K_0^{\frac{d-3}{2}+2\beta}(\eta_1,\eta_2)
=
0
\]
on $(\setF\times \setG)\setminus \setN$, which would only occur when $s=0$ (the equality condition of \eqref{ineq:keyest wave}).

\section{Analogous results for $(++)$ case}\label{sec:analogous results ++}
Here we note the analogous versions of our main results in the $(++)$ case. Here we observe an interesting phenomenon; the null-form $|\square-(2s)^2|$ instead of $|\square|$ somehow fits into the estimate in this framework, which does not occur in similar discussions for the wave propagators in \cite{BJO16}.
\begin{theorem}\label{thm:main ++}
For $d\ge2$ and $\beta>\frac{3-2d}{4}$, we have the sharp estimate
\begin{align}\label{ineq:main ++}
&\||\square-(2s)^2|^\beta (e^{it\phi_s(D)}fe^{it\phi_s(D)}g)\|_{L^2(\mathbb{R}^{d+1})}^2\\
&\quad\le
\KGpp(\beta,d)
\int_{\mathbb{R}^{2d}}
|\widehat{f}(\eta_1)|^2|\widehat{g}(\eta_2)|^2\phi_s(|\eta_1|)\phi_s(|\eta_2|)
\K_\frac12^{\frac{d-2}{2}+2\beta}(\eta_1,\eta_2)
\,\d\eta_1\d\eta_2,\nonumber
\end{align}
where
\[
\KGpp (\beta,d)
:=
2^{\frac{-5d+1}{2}+2\beta}\pi^{\frac{-5d+1}{2}}\frac{\Gamma(\tfrac{d-1}{2})}{\Gamma(d-1)}.
\]
Moreover, the constant $\KGpp(\beta,d)$ is sharp.
\end{theorem}
One may note that by invoking the Legendre duplication formula; $\Gamma(z)\Gamma(z+\tfrac12)=2^{1-2z}\pi^\frac12\Gamma(2z)$,
\[
\KGpp(\beta,d)=\frac{2^{\frac{-7d+5}{2}+2\beta}\pi^{\frac{-5d+2}{2}}}{\Gamma(\tfrac d2)}
\]
holds, which is the same constant introduced in \cite{BJO16}.

An extra symmetry in the $(++)$ case allows \eqref{ineq:main ++} a wider range of $\beta$ than $\beta>\frac{1-d}{4}$ for \eqref{ineq:main}. Indeed, the condition $\beta>\frac{1-d}{4}$ is no longer imposed because of the form of the sharp constant $\KG_{(++)}(\beta,d)$, and the alternative lower bound of $\beta$ emerges from the kernel $\K_\frac12^{\frac{d-2}{2}+2\beta}$. To see this, let us consider an extremiser $f=g=f_1$ in \eqref{extremisers} and after applying polar coordinates to the right-hand side of \eqref{ineq:main ++} (without the constant) we obtain
\begin{align}\label{f:RHS}
&\iint
e^{-(\phi_s(r_1)+\phi_s(r_2))}(\phi_s(r_1)\phi_s(r_2))^{-1}(r_1r_2)^{d-1}\\
&\qquad\times
\int_{-1}^1
\frac{(\phi_s(r_1)\phi_s(r_2)-s^2-r_1r_2\lambda)^{\frac{d-2}{2}+2\beta}}{(\phi_s(r_1)\phi_s(r_2)+s^2-r_1r_2\lambda)^\frac12}(1-\lambda^2)^\frac{d-3}{2}\,\d\lambda\d r_1\d r_2.\nonumber
\end{align}
By \eqref{ineq:keyest non-wave}, one may observe that \eqref{f:RHS} is bounded (up to some constant) by
\[
s^{-3}
\iint
e^{-(\phi_s(r_1)+\phi_s(r_2))}(\phi_s(r_1)\phi_s(r_2))^{\frac{3d-4}{2}+2\beta}
\int_{-1}^1
\left(
1-\frac{r_1r_2}{\phi_s(r_1)\phi_s(r_2)-s^2}\lambda
\right)^{\frac{d-2}{2}+2\beta}
(1-\lambda^2)^\frac{d-3}{2}\,\d\lambda\d r_1\d r_2
\]
and finite whenever $\beta>\frac{3-2d}{4}$ by applying Lemma \ref{lem:monotonicity}.\\

In order to state the various results which follow from Theorem \ref{thm:main ++}, here we introduce the following constant for the wave regime
\[
\Fpp(\beta,d)
=
2^{-1+4\beta}\pi^{\frac{-d+1}{2}}\frac{\Gamma(d-2+2\beta)}{\Gamma(\frac{3d-5}{2}+2\beta)}.
\]

\begin{corollary}\label{cor:biest radial ++}
Let $d\ge2$, $\beta\geq\frac{2-d}{4}$. Then, there exists a constant $C>0$ such that
\begin{align}\label{ineq:biest D_-D_+ ++}
\||\square-(2s)^2|^{\beta}(e^{it\phi_s(D)}f e^{it\phi_s(D)}g)\|_{L^2(\mathbb{R}^{d+1})}
\le
C
\|\phi_s(D)^{\frac{d-1}{4}+\beta}f\|_{L^2(\mathbb{R}^d)}
\|\phi_s(D)^{\frac {d-1}{4}+\beta}g\|_{L^2(\mathbb{R}^d)}
\end{align}
holds whenever $f$ and $g$ are radially symmetric.
Moreover, for $\beta\in[\frac{2-d}{4},\tfrac{3-d}{4}]\cup[\tfrac{5-d}{4},\infty)$, the optimal constant in \eqref{ineq:biest D_-D_+ ++} for radially symmetric $f$ and $g$ is $\Fpp(\beta,d)^\frac12$, but there does not exist a non-trivial pair of functions $(f,g)$ that attains equality.
\end{corollary}

We remark that, following the argument in Section \ref{sec:arg. on gap}, once we apply Theorem \ref{thm:main ++} as a first step, it is not possible to obtain the constant $\F_{(++)}(\beta,d)^\frac12$ for $\beta\in(\frac{3-d}{4},\frac{5-d}{4})$.

\begin{corollary}\label{cor:wave ++} 
Let $d\ge2$. 
\begin{enumerate}[(i)]
\item
The estimate
\begin{equation}\label{ineq:corollaries++}
\||\square-(2s)^2|^\beta (e^{it\phi_s(D)}f)^2\|_{L^2(\mathbb{R}^{d+1})}
\le
C
\|\phi_s(D)^\alpha f\|_{L^2(\mathbb R^d)}^2
\end{equation}
holds with the optimal constant $C=\Fpp(\beta,d)^\frac12$
for $(\alpha,\beta)=(\frac12,\frac{3-d}{4})$ and $(\alpha,\beta)=(1,\frac{5-d}{4})$, but there are no extremisers. Furthermore, when $(\alpha,\beta)=(1,\frac{5-d}{4})$, we have the refined Strichartz estimate
\begin{align*}
\||\square-(2s)^2|^\frac{5-d}{4} (e^{it\phi_s(D)}f)^2\|_{L^2(\mathbb{R}^{d+1})}
\le
\Fpp(\tfrac{5-d}{4},d)^\frac12
\left(
\|\phi_s(D)f\|_{L^2(\mathbb R^d)}^4
-s^2
\|\phi_s(D)^\frac12f\|_{L^2(\mathbb R^d)}^4
\right)^\frac12,
\end{align*}
where the constant is optimal and there are no extremisers. 

\item
The estimate \eqref{ineq:corollaries++} holds with the optimal constant 
$$
C
=
\left(
2^{1-d}\pi^{\frac{-d+1}{2}}\frac{\Gamma(\frac{d-1}{2})}{\Gamma(\frac d2)}s^{-1}
\right)^\frac12
$$
for $(\alpha,\beta)=(\frac12,\frac{2-d}{4})$, but there are no extremisers. Furthermore, when $(\alpha,\beta)=(1,\frac{4-d}{4})$, we have the refined Strichartz estimate
\begin{align*}
&\||\square-(2s)^2|^\frac{4-d}{4} (e^{it\phi_s(D)}f)^2\|_{L^2(\mathbb{R}^{d+1})}\\
&\quad\le
\left(
2^{2-d}\pi^{\frac{-d+1}{2}}\frac{\Gamma(\frac{d-1}{2})}{\Gamma(\frac{d+2}{2})}s^{-1}
\right)^\frac12
\left(
\|\phi_s(D)f\|_{L^2(\mathbb R^d)}^4
-s^2
\|\phi_s(D)^\frac12f\|_{L^2(\mathbb R^d)}^4
\right)^\frac12,
\end{align*} 
where the constant is optimal and there are no extremisers.
\end{enumerate}
\end{corollary}
Theorem \ref{thm:main ++} follows from adapting the calculation by Jeavons \cite{Jv14} without any major difficulty. Indeed, he has computed by the Cauchy--Schwarz inequality that 
\[
|\widetilde{(uv)}(\tau,\xi)|^2
\leq
\frac{J^0(\tau,\xi)}{(2\pi)^{2d-2}}\int_{(\mathbb R^d)^2}|F(\eta_1,\eta_2)|^2\delta{\tau-\phi_s(|\eta_1|)-\phi_s(|\eta_2|)\choose\xi-\eta_1-\eta_2}\,\d\eta_1\d\eta_2.
\]
Here, $u(t,x)=e^{it\phi_s(\sqrt{-\Delta})}f(x)$, $v(t,x)=e^{it\phi_s(\sqrt{-\Delta})}g(x)$, $$F(\eta_1,\eta_2)=\widehat f(\eta_1)\widehat g(\eta_2)\phi_s(|\eta_1|)^\frac12\phi_s(|\eta_2|)^\frac12,$$
and $J^\beta$ is given by \eqref{eq:reformation of J}.
In terms of the Lorentz invariant measure $\d\sigma_s(t,x)=\frac{\delta(t-\phi_s(|x|))}{\phi_s(|x|)}\,\d x\d t$, $J^0(\tau,\xi)$ is also written as
\[
J^0(\tau,\xi)
=
\sigma_s*\sigma_s(\tau,\xi).
\]
Invoking Lemma 1 in \cite{Jv14}, we have
\[
J^0(\tau,\xi)
=
\frac{|\mathbb S^{d-1}|}{2^{d-2}}\frac{(\tau^2-|\xi|^2-(2s)^2)^{\frac{d-3}{2}}}{(\tau^2-|\xi|^2)^\frac12}
\]
so that
\begin{align*}
&\||\square-(2s)^2|^\beta (e^{it\phi_s(D)}fe^{it\phi_s(D)}g)\|_{L^2(\mathbb{R}^{d+1})}^2\\
&\quad=
(2\pi)^{-(d+1)}\int_{\mathbb R^{d+1}}|\tau^2-|\xi|^2-(2s)^2|^{2\beta}|\widetilde{(uv)}(\tau,\xi)|^2\,\d\xi\d\tau\\
&\quad\leq
\frac{2^{\frac{-7d+3}{2}+2\beta}\pi^{-2d+1}}{\Gamma(\tfrac d2)}|\mathbb S^{d-1}|
\int_{(\mathbb R^d)^2}|\widehat f(\eta_1)|^2|\widehat g(\eta_2)|^2\phi_s(|\eta_1|)\phi_s(|\eta_2|)\K_\frac12^{\frac{d-2}{2}+2\beta}(\eta_1,\eta_2)\,\d\eta_1\d\eta_2,
\end{align*}
which is what we desired.\qed


\end{document}